\documentclass[noams]{compositio}
\usepackage[mathscr]{eucal}
\usepackage{palatino, mathpazo, amsfonts, mathrsfs, amscd}
\usepackage{amssymb, amsmath}
\usepackage{stackrel}
\usepackage{bm}

\usepackage{tikz-cd}
\usepackage{tikz}
\usetikzlibrary{arrows}
\newtheorem{theorem}{Theorem}[section]
\newtheorem{lemma}[theorem]{Lemma}

\newtheorem{proposition}[theorem]{Proposition}
\newtheorem{corollary}[theorem]{Corollary}

\theoremstyle{remark}
\newtheorem{remark}[theorem]{Remark}

\newtheorem{definition}[theorem]{Definition}
\newtheorem{assumption}[theorem]{Assumption}

\numberwithin{equation}{subsection}

\newcommand{\sslash}{\mathbin{/\mkern-6mu/}}

\newcommand{\spec}{\operatorname{Spec}}

\newcommand{\ul}{\underline}
\newcommand{\wt}{\widetilde}

\newcommand{\cc}[1]{\mathcal{#1}}  % mathcal
\newcommand{\f}[1]{\mathfrak{#1}}  % mathcal

\newcommand{\CC}{\mathbb{C}}
\newcommand{\ZZ}{\mathbb{Z}}

\newcommand{\NN}{\mathbb{N}}
\newcommand{\II}{\mathbb{I}}

\newcommand{\PP}{\mathbb{P}}
\newcommand{\QQ}{\mathbb{Q}}

\newcommand{\RR}{\mathbb{R}}

\newcommand{\bq}{\bm{q}}
\newcommand{\bbr}{\bm{r}}

\newcommand{\bt}{{\bm{t}}}

\newcommand{\omlog}{\omega_{\cC, \op{log}}}

\newcommand{\ii}{\mathbb{1}}

\newcommand{\cH}{\cc H}
\newcommand{\cV}{\cc V}

\newcommand{\cC}{\cc C}
\newcommand{\cP}{\cc P}

\DeclareMathOperator{\tot}{tot}

\DeclareMathOperator{\eff}{Eff}

\DeclareMathOperator{\Sym}{Sym}

\newcommand{\br}[1]{\left\langle#1\right\rangle}  % angle brackets
\newcommand{\brr}[1]{\left\langle\!\left\langle#1\right\rangle\!\right\rangle}

\newcommand{\op}[1]{\operatorname{#1}}

\renewcommand{\contentsname}{Contents\\{\footnotesize\normalfont(A table
of contents should normally not be included)}}
\begin{document}

\title{Towards a mirror theorem for GLSMs}
\author{Mark Shoemaker}
\email{mark.shoemaker@colostate.edu}
\address{Colorado State University\\Department of Mathematics\\1874 Campus Delivery\\Fort Collins, CO, USA, 80523-1874
}
%
%\dedication{A dedication can be included here.}
\thanks{This work was partially supported by NSF grant DMS-1708104 and Simons Foundation Travel Grant 958189.
}

\begin{abstract}
We propose a method for computing generating functions of genus zero invariants of a gauged linear sigma model $(V, G, \theta, w)$.  We show that certain derivatives of $I$-functions of quasimap invariants of $[V\sslash_\theta G]$ produce $I$-functions (appropriately defined) of the gauged linear sigma model.  When $G$ is an algebraic torus we obtain an explicit formula for an $I$-function, and check that it agrees with previously computed $I$-functions in known special cases. Our approach is based on a new construction of these invariants which applies whenever the evaluation maps from the moduli space are proper, and includes insertions from \emph{light} marked points, which may collide with each other and with base points.
\end{abstract}

\maketitle

%\vspace*{6pt}\tableofcontents  % for this guide only.
% A table of contents should normally not be included
\section{Introduction}
A \emph{gauged linear sigma model} (GLSM) consists, roughly, of a graded complex vector space $V$ acted on by a linear reductive group $G$, together with a choice of character $\theta \in \widehat G$, and a $G$-invariant function $w\colon V \to \CC$.  The data $(V, G, \theta, w)$ together define a GIT quotient $Y := [V\sslash_\theta G]$ and a function $w\colon Y \to \CC$ on the quotient.

GLSMs arise naturally in a number of contexts, for instance as the mirrors to Fano manifolds \cite{G2, HV}, as examples of noncommutative crepant resolutions \cite{Sha, Asp}, and as intermediate spaces appearing in wall crossing and correspondence results \cite{Wi2, Or, FK2}.  GLSMs provide a broad setting in which it is possible to define an enumerative curve counting theory.  They simultaneously generalize the FJRW theory of an isolated singularity (\emph{affine phase} GLSMs) as well as the Gromov--Witten theory of complete intersections (\emph{geometric phase} GLSMs).

Due to their increasingly prominent role in enumerative geometry and mathematical physics, a significant effort has been made over the last decade to provide a rigorous mathematical theory of curve counting invariants for general GLSMs.  
By the efforts of Fan--Jarvis--Ruan \cite{FJR2}, Tian--Xu \cite{TX0, TX2, TX3, TX}, the author together with Ciocan-Fontanine--Favero--Gu\'er\'e--Kim \cite{CFGKS}, Chang--Kiem--Li \cite{CKL, KL}, and Favero--Kim \cite{FKim}, GLSM invariants have now been defined in a variety of contexts and levels of generality.  As evidenced by the number and variety of approaches, the landscape is still developing rapidly.  The relationship between these different definitions is largely unexplored.

Despite the dramatic recent progress in defining GLSM invariants, with the exception of affine phases, geometric phases, and a class of Picard-rank one 
hybrid model phases \cite{ClRo}, there have been very few computations of these new invariants.  In this paper we develop a method for computing genus zero invariants for a large class of GLSMs.  We then define and compute new generating functions of GLSM invariants, which may be viewed as analogous to the big $I$-functions arising in mirror theorems in Gromov--Witten theory \cite{CCIT2, CKbig, CCKbig}.

Our method is based on a direct comparison between the genus zero GLSM invariants of $(V, G, \theta, w)$ and the Gromov--Witten (or more precisely quasimap) invariants of $[V \sslash_\theta G]$.  We prove that $I$-functions for the GLSM invariants of $(V, G, \theta, w)$ arise as derivatives of $I$-functions for the Gromov--Witten invariants of $[V\sslash_\theta G]$.
We expect this formulation to be useful in transporting known results in Gromov--Witten theory to the new setting of GLSMs.  We will explore such applications in future works.

\subsection{$0^+$-stability and light marked points}

In a series of papers, Ciocan--Fontanine-Kim \cite{CCKtoric, CKbig} together with Maulik \cite{CKM} and Cheong \cite{CCKbig} define $\epsilon$-stable quasimap invariants for GIT stack quotients $[V \sslash_\theta G]$, depending on a parameter $\epsilon \in \RR_{>0}$.  The invariants are defined as integrals over moduli spaces $Q_{h, n}^{\epsilon}([V \sslash_\theta G], d)$ which parametrizes a class of \emph{rational} maps to \\ $[V \sslash_\theta G]$ which may contain basepoints of length at most $1/\epsilon$.  
For $\epsilon$ sufficiently large, the invariants are referred to as $\infty$-stable, and coincide with Gromov--Witten invariants.  In the limit as $\epsilon$ approaches $0$, we refer to the invariants as $0^+$-stable quasimap invariants.  It was observed in \cite{CK1} that $0^+$-stable quasimap invariants are particularly relevant to mirror symmetry---in fact they are expected to coincide precisely with  $B$-model invariants.  

In \cite{FJR2}, Fan--Jarvis--Ruan construct moduli spaces of $\epsilon$-stable LG quasimaps, $LGQ_{h, n}^{\epsilon}([V \sslash_\theta G], d)$, used to define GLSM invariants for $(V, G, \theta, w)$.  These moduli spaces and the corresponding invariants also depend on a parameter $\epsilon$.  
As with quasimaps, the $0^+$-stable GLSM invariants are the most directly relevant to mirror symmetry.  

There is  a second reason to focus on $0^+$-stable GLSM invariants.  For most GLSMs, the $\epsilon$-stable invariants are \emph{only defined for $\epsilon = 0^+$}.   The existence of the moduli space $LGQ_{h, n}^{\epsilon}([V \sslash_\theta G], d)$ for $\epsilon > 0^+$ requires the existence of a \emph{good lift} of $\theta$ (see \cite[Definition~3.2.13]{FJR2}).  This condition turns out to be quite restrictive, especially outside the setting of geometric, affine, and hybrid model phases.  Therefore, a study of general GLSMs must necessarily work with $0^+$-stability.  This is the approach we take.

Quasimap invariants can be generalized further to allow for a second class of marked points 
 on the source curve which may coincide with each other and with base points, but which are still distinct from nodes and regular  marked points.  We call these new markings \emph{light} marked points to distinguish them from the usual (heavy) marked points.

In \cite{CKbig} it was observed that by adding insertions from light markings one can compute \emph{big} $I$-functions---generating functions which encode all of the genus zero Gromov--Witten theory of $[V \sslash_\theta G]$.  Light marked points also feature in the work of Losev--Manin \cite{LM} in the study of pencils of formal flat connections.

The GLSM $I$-functions we define also involve invariants with insertions from light markings.
Defining GLSM invariants with light point insertions appears to be a difficult problem in general, however we show that for the particular invariants of interest to us, a simple definition is possible (see Definition~\ref{d:lightinvts} and Remark~\ref{altcomp}).  

\subsection{Comparison of invariants}

Let $(V, G, \theta, w)$ be a GLSM.  In Gromov--Witten theory, the state space for $[V\sslash_\theta G]$ is given by the Chen--Ruan cohomology $H^*_{CR}([V\sslash_\theta G])$.  We will focus on the \emph{compact type} subspace, defined as the image of the natural map from compactly supported cohomology:
$$H^*_{CR,ct}([V\sslash_\theta G]) := \op{im}(H^*_{CR, cs}([V\sslash_\theta G]) \to H^*_{CR}([V\sslash_\theta G])).$$

The state space of the GLSM $(V, G, \theta, w)$ is given by
$$\cc H(Y, w) := H^*_{CR}(Y, w^{+\infty}),$$
where $w^{+\infty} = (\mathfrak{ Re}(w))^{-1}\left( M, \infty \right)$ for $M$ a sufficiently large real number.  There is a compact type subspace $\cc H_{ct}(Y, w) \subseteq \cc H(Y, w)$ as well, defined roughly as the space spanned by classes Poincar\'e dual to proper substacks supported in $w^{-1}(0)$ (Definition~\ref{d:ctglsm}).  There exists a natural map between the GLSM and Gromov--Witten compact type state spaces:
$$\phi^w: \cc H_{ct}(Y, w) \to H^*_{CR,ct}([V\sslash_\theta G]).$$

Our first result is a comparison between two-pointed genus zero GLSM and quasimap invariants with compact type insertions.
\begin{theorem}[Corollary~\ref{c:2pt}]  
Given classes $\gamma_1, \gamma_2 \in \cc H_{ct}(Y, w)$ and $d \in \op{Eff}([V/G])$, 
$$\br{\gamma_1 \psi^{k_1}, \gamma_2\psi^{k_2}}_{0,2,d}^{(Y, w)} = \br{\phi^w(\gamma_1)\psi^{k_1}, \phi^w(\gamma_2)\psi^{k_2}}_{0,2,d}^{Y},$$
where the left hand side is the ($0^+$-stable) GLSM invariant and the right hand side is the ($0^+$-stable) quasimap invariant.
\end{theorem}

We can extend the above result to include light marked points.  For the case of quasimaps, because light markings may coincide with base points, insertion classes are pulled back from the cohomology of the stack quotient: $H^*([V/G])$.  In the case of GLSMs, one may define light point insertions as coming from $H^*([ V^J/(G/\langle J\rangle)], w^{+\infty})$, where $J$ is a distinguished element of $G$ determined by the grading on $V$ (Definition~\ref{d:glsm0}).  
We prove the following:
\begin{theorem}[Theorem~\ref{t:complight}] \label{t:t1}
There exists a map 
$$\tau\colon H^*([V/G]) \to H^*([ V^J/(G/\langle J\rangle)], w^{+\infty})$$ such that, for $\gamma_1, \gamma_2 \in \cc H_{ct}(Y, w)$ and $\alpha_1, \ldots, \alpha_k \in H^*([V/G])$, 
\[\br{\gamma_1 \psi^{k_1}, \gamma_2\psi^{k_2}| \tau(\alpha_1), \ldots, \tau(\alpha_k)}_{0,2|k,d}^{(Y, w)} =  \br{\phi^w(\gamma_1)\psi_1^{k_1},\phi^w(\gamma_2)\psi_2^{k_2}| \alpha_1, \ldots, \alpha_k}_{0,2|k,d}^{Y},\]
where the left (resp. right)-hand side denotes the $0^+$-stable GLSM (resp. quasimap) invariant with two heavy point insertions and $k$ light point insertions.
\end{theorem}
While the context of this theorem may seem rather restrictive (genus zero with two heavy marked points with compact type insertions), these are \emph{precisely} the set of invariants needed to construct  big $I$-functions for GLSMs (restricted to the compact type subspace).  Theorem~\ref{t:t1}  provides a method for computing these GLSM generating functions.  
\subsection{A mirror theorem for GLSMs}
What is meant precisely by a mirror theorem in Gromov--Witten theory varies across the literature.  Here we state our main results, and explain in what sense they can be viewed as giving a mirror theorem for GLSMs.

The original mirror theorem \cite{G1, LLY} was a correspondence between the $J$-function and $I$-function of the quintic threefold $Q$.  The former is a generating function of Gromov--Witten invariants, while the latter was defined in terms of period integrals on the mirror, and has a combinatorial expression in terms of hypergeometric series.  The original mirror theorem states that $J^Q$ is equal to $I^Q$ after a change of variables and scaling by an invertible function in the Novikov variables.  

A more modern formulation of the mirror theorem uses the language of Givental's symplectic formalism and applies to more general targets \cite{Bro, CCIT2} and to \emph{big} $I$-functions \cite{CKbig, CCKbig, Wang, Zhou}.  In this context, a mirror theorem for $Y$ amounts to finding an explicit function $I^Y$ (known as an $I$-function) which lies on the \emph{overruled Lagrangian cone} $\cc L^Y$.  Given such a statement, one can then recover Gromov--Witten invariants of $Y$ from $I^Y$  via a process of Birkhoff factorization. 

A fundamental result from the theory of  quasimaps is that  $I$-functions for $Y$ may be defined directly in terms of $0^+$-stable quasimap invariants \cite{CK1, CKbig, Zhou}.  More precisely, $I$-functions for $Y$ often take the form
\begin{equation}\label{e:IY1} I^Y = S^{Y}(q, f(\bt), z)P(q, \bt, z),\end{equation}
where $S^Y$ is the $0^+$-stable quasimap generating function in genus zero with two heavy points and arbitrarily many light points:
\begin{equation}\label{e:Sq0} S^Y(q, \bt, z)(\gamma) := \sum_{i \in \tilde I} \tilde \gamma_i \brr{ \frac{\tilde \gamma^i}{z- \psi}, \gamma}_{0, 2}^{Y}(\bt),\end{equation}
$f \in H^*([V/G])[[t^i]]$ is a formal  cohomology-valued function satisfying $f(0) = 0$, and $P(q, \bt, z) \in H^*_{CR}(Y)[[q]][[t^i]][z]$ is some cohomology-valued function with only positive powers of $z$ (Definition~\ref{d:BigI}).

In this paper we take \eqref{e:IY1} as the \emph{definition} of an $I$-function for $Y$.  This definition has the advantage of easily generalizing to the GLSM setting, even in cases where the Lagrangian cone for the GLSM is not defined (i.e. when $\infty$-stable GLSM invariants do not exist).  We define an $I$-function for the GLSM $(V, G, \theta, w)$ to be a function $I^{(Y, w)}$ taking the form
\begin{equation} I^{(Y, w)} = S^{(Y, w)}(q, f(\bt), z)P(q, \bt, z) \end{equation}
where $S^{(Y, w)}(q, f(\bt), z) $ is the $0^+$-stable GLSM generating function in genus zero with two heavy points and arbitrarily many light points analogous to \eqref{e:Sq0}, $f \in H^*([ V^J/(G/\langle J\rangle)], w^{+\infty})[[t^i]]$ is a formal  cohomology-valued function satisfying $f(0) = 0$, and $P(q, \bt, z) $ lies in $ \cc H(Y, w)[[q]][[t^i]][z]$  (Definition~\ref{d:BigIGLSM}).
If $P(q, \bt, z) $ is sufficiently general, 
one can, in theory, recover $S^{(Y, w)}(q, f(\bt), z)$ from an $I$-function $I^{(Y, w)}$ via Birkhoff factorization.   

A particular $I$-function for $Y$, denoted $ \II^Y(q, \bt, z)$, was defined in \cite{CKbig, CCKbig} using localization on the quasimap graph moduli space
(see \eqref{e:I} for the precise definition). 
Our  main result provides a way of obtaining an $I$-function for the GLSM $(V, G, \theta, w)$ in terms of    $ \II^Y(q, \bt, z)$ for $Y = [V\sslash_\theta G]$.
\begin{theorem}[Theorem~\ref{t:dI}] \label{t:dI0}
  Given a class $\rho \in H^*([V/G])$, if
$z \partial_\rho \II^Y(q, \bt, z)$ lies in cohomology of compact type, then $\sigma^w\left(z \partial_\rho \II^Y(q, \bt, z)\right)$ is a GLSM $I$-function for $(V,G, \theta, w)$.
\end{theorem}

The above theorem is completely general.  However, in  the abelian case we  can say more.  Let $(V, G, \theta, w)$ be a GLSM for which $G$ is a connected torus.  
Through a nontrivial application of Theorem~\ref{t:dI0},  we obtain an explicit formula for a GLSM  $I$-function under a mild condition on the action of $G$ on $V$.  See Section~\ref{s:GF} for an explanation of notation.
\begin{theorem}[Corollary~\ref{c:Ifctn}]\label{t:tor0} Let $(V, G, \theta, w)$ be a GLSM for which $G$ is a torus ($\cong (\CC^*)^k$).  
Suppose the coordinate $x_i$ on $V = \spec(\CC[x_i]_{i=1}^r)$ has weight $c_i$ with respect to the grading on $V$.  
Let $\CC^*_R$ denote the one-dimensional torus, acting on $V$ with weights $(c_1, \ldots, c_r)$, and let $\Gamma := G \cdot \CC^*_R$.
If $\left(\CC[x_i]_{1\leq i\leq r}\right)^\Gamma = \CC$, then 
\begin{align} \nonumber
\II^{(Y, w)}(\bt, q, z) :=& \sum_{d \in\eff(V, G, \theta)}q^d \exp \left(  \frac{1}{z} \sum_{j=1}^l t^j p_j\left(\eta_k + z\langle d, \eta_k\rangle\right) \right) \\
&\sigma^w\left(  
\frac{\prod_{i| c_i \neq 0,  \langle d, \rho_i\rangle\leq 0} \prod_{ \langle d, \rho_i\rangle \leq \nu\leq 0} (\rho_i + ( \langle d, \rho_i\rangle - \nu)z)}
{\prod_{i| c_i \neq 0, \langle d, \rho_i\rangle> 0} \prod_{ 0 < \nu < \langle d, \rho_i\rangle} (\rho_i + ( \langle d, \rho_i\rangle - \nu)z)}  \nonumber \right. \\
&\left. 
 \frac{\prod_{ i| c_i = 0, \langle d, \rho_i\rangle< 0} \prod_{ \langle d, \rho_i\rangle \leq \nu< 0} (\rho_i + ( \langle d, \rho_i\rangle - \nu)z)}
{\prod_{ i| c_i = 0, \langle d, \rho_i\rangle> 0} \prod_{ 0 \leq \nu < \langle d, \rho_i\rangle} (\rho_i + ( \langle d, \rho_i\rangle - \nu)z)} \ii_{g_d^{-1}} \right) \label{e:IYW0}
\end{align}
is a GLSM $I$-function for $(Y, w)$.
\end{theorem}

In Section~\ref{s:EXS} we compute  $\II^{(Y, w)}$ in the particular cases of FJRW theory (including non-concave examples), Gromov--Witten theory (including non-convex examples), and hybrid theories.  We show, in each case, that \eqref{e:IYW0} agrees with previously computed $I$-functions.

Mirror constructions for toric GLSMs have been proposed by Gross--Katzarkov--Ruddat \cite{GKR} and Clarke \cite{Clarke} based on a duality of polyhedral cones.  A proposal for mirrors of nonabelian GLSMs has recently been put forth by Gu--Parsian--Sharpe and  Gu--Guo--Sharpe \cite{GPS, GGS}.  As is the case in Gromov--Witten theory, we expect the components of $\II^{(Y, w)}(\bt, q, z)$ to coincide with oscillatory integrals of the mirror GLSM.  Integral formulas for
$I$-functions for certain GLSMs have been obtained in a physics context by Knapp--Romo--Scheidegger \cite{KRS}.  Since the original posting of this paper,  an alternative  computation of GLSM $I$-functions using loop spaces has been given by Aleshkin--Liu in \cite{AL} to prove a Higgs--Coulomb correspondence.  It will be interesting to compare these different approaches.

\subsection{Plan of paper}

In Section~\ref{s:setup} we define a gauged linear sigma model and describe its associated state space.  In Section~\ref{s:comptyp} we propose a definition of GLSM invariants in the particular case when the evaluation maps are proper and the insertions are of compact type.  This section also contains a comparison between our definition and the original definition of \cite{FJR2b}.  Sections~\ref{s:g02mp} and~\ref{s:lp} consider the special case of genus zero GLSM invariants with exactly two heavy marked points, where we prove our comparison result relating  GLSM invariants to associated quasimap invariants.  In Section~\ref{s:GF} we define $I$-functions for a general GLSM and show how to compute these generating functions from known $I$-functions in Gromov--Witten theory.  Finally, in Section~\ref{s:EXS} we compare our formula with previous computations in the case of affine, geometric, and hybrid model GLSMs.

\section{GLSM setup and the state space}\label{s:setup}
In this section we define a gauged linear sigma model and set notation.  We also describe the various state spaces which will be used throughout the paper.

\subsection{Setup}
\begin{definition}\label{d:glsm0}
A \emph{gauged linear sigma model} (GLSM) $(V, G, \theta, w)$ consists of the following data:
\begin{enumerate}
\item a finite-dimensional $\ZZ$-graded complex vector space, $$V=\oplus_{i \in\ZZ_{\geq 0}}V_i,$$ where the grading is induced by the action  of a one-dimensional torus 
$\CC^*_R  \subseteq GL(V)$ (called the \emph{R-charge});
\item an action by a linearly reductive group $G  \subseteq GL(V)$;
\item a choice of character $\theta \in \widehat G_\QQ := \mathrm{Hom}(G,\CC^*) \otimes_\ZZ \QQ$; 
\item a $G$-invariant polynomial function $w \colon  V \to \CC$, homogeneous of degree $d_w > 0$ with respect to the grading action $\CC^*_R$,
\end{enumerate}
satisfying the following requirements:
 \begin{itemize} 
\item $G$ and $\CC^*_R$ commute: $g\cdot \lambda = \lambda \cdot g$ for all $g \in G$ and $\lambda \in \CC^*_R$;
\item the intersection $G \cap \CC^*_R$ is cyclic of order $d_w$ with generator $J$,
the diagonal element in $GL(V)$ which acts on $V$ by multiplying $V_n$ by $e^{2 \pi i n/d_w}$; 
\item there are no strictly semistable points for the linearization of the $G$-action on $V$ given by $\theta$:
\begin{equation}\label{nostrictsemistablepts}
V^{ss}(\theta)=V^s(\theta);
\end{equation}
\item
the critical locus $Z(dw) \subset Y$ is a non-empty proper over $\spec \CC$;
\item the graded pieces $V_i$ of $V$ are empty for $i<0$ and $i>d_w$:  $$V = \oplus_{0 \leq i \leq d_w} V_i;$$
\item the group $\CC^*_R$ preserves the $\theta$-semistable  locus $V^{ss}(\theta)$.
\end{itemize}
We denote by $\Gamma$ the group $G \cdot \CC^*_R.$
\end{definition}

\begin{remark}
In some contexts (\cite{FJR2, FKim}), a GLSM is also required to have a \emph{good lift} of $\theta$, that is, a character $\nu \in \widehat \Gamma$ such that $\nu|_G = \theta$ and $V^{ss}(\nu) = V^{ss}(\theta)$.  A good lift is necessary for defining $\epsilon$-stable GLSM invariants for $\epsilon > 0^+$.  In this paper we work exclusively with $0^+$-stable invariants, so we do not require a good lift.  Indeed one of the main goals of this paper is to compute GLSM invariants in the absence of a good lift.
\end{remark}

Let $Y$ denote the GIT stack quotient
$$Y  := [V \sslash_\theta G].$$ 
The function $w\colon V \to \CC$ descends to a function on $Y$, called the \emph{potential function}. By abuse of notation we will also denote the potential function by $w$.  

Given a smooth orbifold $Y$ and a function $w: Y \to \CC$, the pair $(Y, w)$ is known as a \emph{Landau--Ginzburg (LG) model}.  A \emph{map of LG models} $$j': (Y, w) \to (Y', w')$$ is a mophism $j: Y \to Y'$ such that 
$$w' \circ j = w.$$

We say that the GLSM $(V, G, \theta, w)$ \emph{represents} the LG model $(Y, w)$ if \\ $[V \sslash_\theta G] = Y$ and the potential $w: Y \to \CC$ is induced by the $G$-equivarient function on $V$ of the same name.  We note that two different GLSMs can represent the same LG model.  In this case one expects a relationship between the associated GLSM invariants.  By abuse of notation we will sometimes speak of the GLSM invariants of $(Y, w)$ when the particular GLSM $(V, G, \theta, w)$ representing $(Y, w)$ has been fixed.

We record for future use the following map of exact sequences
\begin{equation}\label{e:xi}
\begin{tikzcd}
0 \ar[r] & \langle J \rangle \ar[r] \ar[d]&  \CC^*_R \ar[r, "(-)^{d_w}"] \ar[d] & \CC^* \ar[r] \ar[d, "\op{id}"]& 0 \\
0 \ar[r] &G \ar[r] & \Gamma \ar[r, "\xi"] & \CC^* \ar[r] & 0
\end{tikzcd}
\end{equation}
where the homomorphism $\xi$ is defined by $\xi(g \cdot \lambda) = \lambda^{d_w}$ for $g \in G$, $\lambda \in \CC^*_R$.

\subsection{The state space}
In this section we introduce the state spaces used for GLSM and quasimap invariants.  

\begin{remark}
The potential $w\colon Y \to \CC$ pulls back to $IY$, the inertia stack of $Y$. We will also denote this pullback by $w$.  Unless otherwise specified, $H_*$ will denote Borel--Moore homology, and all (co)homology groups will have coefficients in $\CC$.
\end{remark}

\begin{definition}
The \emph{GLSM state space} of $(V, G, \theta, w)$ is defined as
$$\cc H(Y, w) := H^*_{CR}(Y, w^{+\infty}) = H^*(IY, w^{+\infty}),$$
where $w^{+\infty} := (\mathfrak{ Re}(w))^{-1}\left( M, \infty \right)$ for $M$ a sufficiently large real number.

The (Gromov--Witten theory) state space of $Y$ is given by
$$H^*_{CR}(Y) :=  H^*(IY).$$
\end{definition}
There is a map 
$$\phi^w\colon \cc H(Y, w) \to H^*_{CR}(Y),$$
defined by pullback via the inclusion of pairs $(IY, \emptyset) \hookrightarrow (IY, w^{+\infty})$.
The kernel of this map is sometimes referred to as the \emph{broad sectors}.

Consider a component $Y_g$ of $IY$, with $g$ a representative of a conjugacy class of $G$.  
Let $j\colon Z \hookrightarrow Y_g$ be a smooth and proper substack of $Y_g$.  Assume that $w|_Z = 0$.  In this case there exists a map of LG models:
$$ j'\colon (Z, 0) \hookrightarrow (Y_g, w) \hookrightarrow (IY, w).$$

  By capping with the fundamental class $\mu_{Y_g} \in H_*(Y_g)$, we obtain an isomorphism as in \cite[Section~19.1]{Fulton}: \begin{equation}\label{e:BMduality}
H^*(Y_g, Y_g\setminus Z) \xrightarrow{\cong} H_*(Z).\end{equation}
The pushforward $j_*\colon H^*(Z) \to H^*_{CR}(Y)$ may be defined as the composition
$$H^*(Z) \xrightarrow{\cong} H_*(Z) \xrightarrow{\cong} H^*(Y_g, Y_g\setminus Z) \to H^*(Y_g) \subset H^*_{CR}(Y),$$
where the first map is Poincar\'e duality, the second is the inverse to \eqref{e:BMduality}, and the third is the natural pullback of inclusion of pairs.

By assumption on $Z$, we have the inclusion $w^{+\infty} \subset IY \setminus Z$.  This allows the following definition:
\begin{definition}\label{d:BMpush}
Define the pushforward $$j'_*\colon H^*(Z) \to \cc H(Y, w)$$ as the composition
$$H^*(Z) \xrightarrow{\cong} H_*(Z) \xrightarrow{\cong} H^*(Y_g, Y_g\setminus Z) \to H^*(Y_g, w^{+\infty}) \subset \cc H(Y, w).$$
\end{definition}
\begin{definition}\label{d:ctglsm}  
Define the \emph{compact type GLSM state space} $\cc H_{ct}(Y, w) \subset \cc H(Y, w)$ as the subspace spanned by classes of the form $j'_*(\alpha)$ for $\alpha \in H^*(Z)$ and $Z$ a smooth and proper substack of $IY$ lying in $w^{-1}(0)$.
\end{definition}

Let $$\phi^{cs}\colon H^*_{CR, cs}(Y) \to H^*_{CR}(Y),$$
denote the natural map from 
 compactly supported Chen--Ruan cohomology $H^*_{CR, cs}(Y)$ to Chen--Ruan cohomology.
\begin{definition}\label{d:ctgw}
Define the \emph{compact type Gromov--Witten theory state space} $H^*_{CR,ct}(Y)$ as the image $\op{im}(\phi^{cs})  \subset H^*_{CR}(Y)$.  
\end{definition}

We observe the following:
\begin{lemma}\label{l:jfactors}
For $Z$ a smooth, proper substack of $w^{-1}(0)$ in $IY$, the pushforward $j_*\colon H^*(Z) \to H^*_{CR}(Y)$ factors as
$$j_* = \phi^w \circ j'_*.$$
\end{lemma}
\begin{proof}
The result is immediate from the definitions above and the obvious commutativity of
\[
\begin{tikzcd}
H^*(Y_g, Y_g\setminus Z) \ar[dr] \ar[r] & H^*(Y_g, w^{+\infty}) \ar[d] \\
& H^*(Y_g).
\end{tikzcd}
\]
\end{proof}

If $j\colon Z \hookrightarrow IY$ is the inclusion of a smooth proper substack, then $j_*(\alpha)$ lies in $H^*_{CR,ct}(Y)$.  By the lemma above, this implies that 
$$ \phi^w \left(\cc H_{ct}(Y, w)\right) \subseteq H^*_{CR,ct}(Y).$$

\subsection{Compact type cup product and pairing}\label{s:ctcpp}
We recall here some facts about compact type cohomology $H^*_{CR,ct}(Y)$ from \cite{Shoe1}.   Note that in \cite{Shoe1}, the compact type cohomology was referred to as \emph{narrow} cohomology.  In this paper we have adopted the language appearing in \cite{FJR2}.

\begin{definition}
Given a class $\gamma \in H^*_{CR,ct}(Y)$, we say $\wt \gamma \in H^*_{CR, cs}(Y)$ is a \emph{lift} of $\gamma$ if $$\phi^{cs}(\wt \gamma) = \gamma.$$
\end{definition}
In addition to the usual cup products:
\begin{align*}
\cup \colon &H^*_{CR}(Y) \times H^*_{CR}(Y) \to H^*_{CR}(Y),  \\
\cup \colon &H^*_{CR, cs}(Y) \times H^*_{CR}(Y) \to H^*_{CR,cs }(Y),
\end{align*}
there is a cup product
$
\cup_{cs}\colon H^*_{CR, ct}(Y) \times H^*_{CR, ct}(Y) \to H^*_{CR,cs}(Y)
$ 
which multiplies a pair of elements from compact type cohomology to an element of compactly supported cohomology.  It is
given by 
$$\gamma_1 \cup_{cs} \gamma_2 := \wt \gamma_1 \cup \gamma_2$$
where $\gamma_1, \gamma_2 \in H^*_{CR,ct}(Y)$ and $\wt \gamma_1$ is a lift of $\gamma_1$.  One can check that if $\kappa \in \ker(\phi^{cs})$ and $\gamma \in H^*_{CR, ct}(Y)$, then
$\kappa \cup \gamma = 0$ in $H^*_{CR, ct}(Y)$, from which it follows that $\cup_{cs}$ is well-defined (See \cite{Shoe1} for details and a proof in de Rham cohomology when $Y$ is smooth, a similar argument holds in singular cohomology).

Recall the Chen--Ruan pairing $\br{ -,-}^Y$.  For $\wt \gamma_1 \in H^*_{CR, cs}(Y)$ and $ \gamma_2 \in H^*_{CR}(Y)$, 
$$\br{\wt \gamma_1,  \gamma_2}^Y := 
\int_{IY} \wt \gamma_1 \cup I_*( \gamma_2),$$
where $I\colon IY \to IY$ is the natural involution.  By \cite{ChenR1}, this pairing is nondegenerate.
We also define a compact type pairing $\br{ -,-}^{Y, cs}$ on $H^*_{CR,ct}(Y)$ by
$$\br{\gamma_1, \gamma_2}^{Y, cs} := \int_{IY} \gamma_1 \cup_{cs} I_*(\gamma_2),$$
By \cite[Corollary~2.7]{Shoe1}, this pairing is nondegenerate.

It is immediate from the definitions that the product and pairing on $H^*_{CR,ct}(Y)$ are compatible with the map $\phi^{cs}$.  That is, for $\wt \gamma_1, \wt \gamma_2 \in H^*_{CR, cs}(Y)$, 
\begin{equation}\label{e:prodcomp}
\wt \gamma_1 \cup \wt \gamma_2 = \phi^{cs}(\wt \gamma_1) \cup \wt \gamma_2  = \phi^{cs}(\wt \gamma_1) \cup_{cs} \phi^{cs}(\wt \gamma_2),
\end{equation}
and
\begin{equation}\label{e:paircomp}
\langle \wt \gamma_1, \phi^{cs}(\wt \gamma_2)\rangle^Y = \langle \phi^{cs}(\wt \gamma_1) , \phi^{cs}(\wt \gamma_2)\rangle^{Y, cs}
\end{equation}
where $\langle -,-\rangle^Y$ denote the usual Chen--Ruan pairing between $H^*_{CR, cs}(Y)$ and $H^*_{CR}(Y)$.

\section{Compact type GLSM invariants}\label{s:comptyp}
In this section we provide a definition of GLSM invariants with compact type  insertions in the special case that the evaluation maps are proper.  We will see in later sections that this is sufficient for a mirror theorem.

\subsection{Review of quasimap invariants}\label{s:qmap}

We denote the moduli space of genus $h$, $n$-pointed, degree $d$ quasimaps to $Y = [V \sslash_\theta G]$ by $Q_{h,n}(Y, d)$.  In this paper we will only work with $0^+$-stability, so we omit it from the notation.  We give a construction of $Q_{h,n}(Y, d)$ and the corresponding quasimap invariants below.  This particular presentation of the moduli space and virtual class appear in \cite{CCKtoric} in the toric setting.  In this generality, the moduli space and virtual class were originally constructed in \cite{CKM}.

\begin{definition}\label{d:frakG}
Let $\f M_{h,n}$ denote the moduli stack of genus $h$, $n$-pointed, pre-stable orbi-curves. By an orbi-curve, we mean a twisted curve together with a section of each gerbe markings (see \cite[Definition~3.1.1]{CFGKS}).  Let $\Sigma_i \to \f M_{h,n}$ denote the $i$th marked point gerbe and let $\Sigma = \coprod_{i=1}^n \Sigma_i$.  Let $\sigma_i: \Sigma_i \to \f C$ denote the the inclusion of the $i$th marked point.

Let $\f M_{h,n}(BG)$ denote the moduli space of genus $h$, $n$-pointed, prestable orbi-curves $\cC$, together with a principal $G$-bundle $\cP$ such that the induced morphism $[\cP]\colon \cC \to BG$ is representable.  

Define the \emph{degree} of a principal $G$-bundle $\cP \to \cC$ to be the element  $d \in \hom_\ZZ(\widehat G, \QQ)$ such that, for each character $\zeta \in \widehat G$, 
$$d(\zeta) = \deg(\cP \times_G \CC_\zeta),$$
where $\CC_\zeta$ is the representation corresponding to $\zeta$.

For $d \in \hom_\ZZ(\widehat G, \QQ)$, define $\f M_{h,n}(BG, d)$ to be the open and closed substack of $\f M_{h,n}(BG)$ for which the principal bundle has degree $d$.
Denote by $\f M_{h,n}(BG, d)^\theta$ the open substack consisting of pairs $(\cC, \cP)$ such that $\cC$ has no rational tails and, on each irreducible component $\cC '$ of $\cC$ for which $\omlog|_{\cC '}$ is trivial,  the degree 
of $$\cc L_\theta := \cP \times_G \CC_\theta$$ restricted to $\cC '$ is positive.
Let $\pi\colon \f C \to \f M_{h,n}(BG, d)^\theta$ denote the universal curve and let $\cP_G \to \f C$ be the universal principal $G$-bundle.  Denote by  $\cV_G$ the vector bundle $ \cP_G \times_G V$. 
\end{definition}

\begin{definition}\label{d:sections}
Define the \emph{space of sections} of $\cV_G$ to be
\[ \op{tot}(\pi_*\cV_G) := \op{Spec} \left( \Sym \left( \RR^1 \pi_*(\omega_\pi \otimes \cV_G^\vee)\right)\right),\]
lying over $\f M_{h,n}(BG, d)^\theta$.
\end{definition}
Closed points of $\op{tot}(\pi_*\cV_G)$ consist of a marked curve $\cC$ together with a section $u_G \in \Gamma(\cC, \cV_G|_\cC)$, which defines a map $[u_G]\colon \cC \to [V/G]$ to the stack quotient.  The \emph{base points} of $[u_G]$ consist of the preimage of the unstable locus $[u_G]^{-1}( [V^{us}(\theta)/G])$.

\begin{definition}\label{d:qmod}
The stack $Q_{h,n}(Y, d)$ is the open substack of $\op{tot}(\pi_*\cV_G)$ satisfying the condition that the basepoint locus of $[u_G]$ is a finite set, disjoint from the nodes and markings of $\cC$.  It is a Deligne--Mumford stack, proper over $\spec( (\Sym^\bullet V^\vee)^G)$ \cite[Theorem~2.7]{CCKbig}.
\end{definition}

\begin{remark}
In case that $Y$ is an orbifold, there is a slight difference in the definition of the moduli space of stable quasimaps of from Definition~\ref{d:qmod} and that originally defined in \cite{CCKbig} due to our convention that the gerbe markings have sections (Definition~\ref{d:frakG}).  This does not effect the quasimap invariants, due to the correction factor $\bbr$ appearing in Definition~\ref{d:qinvts}.
\end{remark}

The moduli space $Q_{h,n}(Y, d)$ admits a natural relative  obstruction theory over $\f M_{h,n}(BG, d)^\theta$ given by $\RR \pi_*(\cV_G)^\vee.$ This defines a virtual class 
$$[Q_{h,n}(Y, d)]^{vir}$$ by \cite{BF}.
By choosing an appropriate $\pi$-acyclic resolution of $\cV_G$, we can realize this virtual class more concretely.
We make use of the following:
\begin{proposition}\cite[Section~3.4]{CFGKS}\label{p:res1}
There exists a resolution of $\cV_G$ by a two-term complex of vector bundles $\cc A \xrightarrow{\delta} \cc B$ such that:
\begin{itemize}
\item $\cc A$ and $\cc B$ are $\pi$-acyclic;
\item \label{i2} there exists a restriction map $\cc A \to \cV_G|_{\Sigma}$ such that the composition $\cV_G \to \cc A \to \cV_G|_{\Sigma}$ is the usual restriction, and $\pi_*(\cc A) \to \pi_*(\cV_G|_{\Sigma})$ is surjective.
\end{itemize}

\end{proposition}
Let $A = \pi_*(\cc A)$ and $B = \pi_*(\cc B)$.  Let $p_A\colon \op{tot}(A) \to \f M_{h,n}(B\Gamma, d)^\theta$ denote the projection.
There exists an open subset $U \subset \tot(A)$ such that $Q_{h,n}(Y, d)$ can be realized as $\{\beta = 0\} \subset U$ where $\beta = p_A^*(\delta) \circ \op{taut}_A $ is the section of $ E = p_A^*(B)$ induced by $\delta$ (see, for instance, the proof of \cite[Theorem~3.2.1]{CCKtoric}).

The two-term complex $$[E^\vee \xrightarrow{p_A^*(\delta)^\vee} p_A^*(A)^\vee = \Omega^1_{U/ \f M_{h,n}(BG, d)^\theta}]$$ gives a resolution of 
$\RR \pi_*(\cV_G)^\vee$ over $U$, and thus restricts to the relative perfect obstruction theory described above.  
In this case, the virtual class may be realized as a localized top Chern class (See \cite[Section~6]{BF} for details).  
Consider the fiber square
\[
\begin{tikzcd}
Q_{h,n}(Y, d) \ar[r] \ar[d] & U \ar[d, "\beta"] \\
U \ar[r, "0"] & \tot(E),\end{tikzcd}
\]
where $\beta$ and $0$ are the maps induced by the corresponding sections of $E$.
\begin{definition}\label{d:virtq} Define the virtual class $[Q_{h,n}(Y, d)]^{vir}$ to be the localized top Chern class
$$[Q_{h,n}(Y, d)]^{vir}:= e_{loc}(E, \beta) = \op{cl}\left(0^!( [U])\right),$$
where $0^!: A_*(U) \to A_*(Q_{h,n}(Y, d))$ is the Gysin pullback and \\ $\op{cl}: A_*(Q_{h,n}(Y, d)) \to H_*(Q_{h,n}(Y, d))$  is the map, described in \cite[Section~19.1]{Fulton}, which gives the corresponding class in Borel--Moore homology.
\end{definition} 
We take this as the definition.  It coincides with the (image in homology of the) virtual class of \cite{BF} using the relative obstruction theory given above, thus is independent of the resolution $[A \to B]$ chosen in Proposition~\ref{p:res1}.

We define quasimap invariants by integrating against the virtual class.  
We require the following lemma.
\begin{lemma}\label{l:propev}
The evaluation map $ev_i\colon Q_{h,n}(Y, d) \to IY$ is proper for $1 \leq i \leq n$.
\end{lemma}
\begin{proof}

By  \cite[Theorem~2.7]{CCKbig}, $Q_{h,n}(Y, d)$ is a proper Deligne-Mumford stack over the affine quotient $\spec( (\Sym^\bullet V^\vee)^G)$. The map $$Q_{h,n}(Y, d) \to \spec( (\Sym^\bullet V^\vee)^G)$$ factors through $ev_i$, therefore by \cite[\href{https://stacks.math.columbia.edu/tag/01W6}{Lemma 01W6}]{stacks-project}, $ev_i$ is proper.
\end{proof}
For $1 \leq i\leq n$ let $\bbr_i$ be the locally constant function on $Q_{h,n}(Y, d)$ with value given by the order of the  automorphism group at the $i$th marking on each connected component (see \cite[Section 2.1]{Ts}).  Let $\bbr$ denote the product $$\bbr = \prod_{i=1}^n \bbr_i.$$ The quasimap invariant $\br{ \gamma_1 \psi^{k_1}, \ldots, \gamma_n \psi^{k_n}}_{h,n,d}^Y$ is well defined if either:
\begin{itemize}
\item at least one class $\gamma_i$ lies in compactly supported cohomology;
\item at least two classes $\gamma_i, \gamma_j$ lie in compact type cohomology.
\end{itemize}
\begin{definition}\label{d:qinvts}
Given $\gamma_1 \in H^*_{CR, cs}(Y)$ and $\gamma_2, \ldots, \gamma_n \in H^*_{CR}(Y)$, define the quasimap invariant $\br{ \gamma_1 \psi^{k_1}, \ldots, \gamma_n \psi^{k_n}}_{h,n,d}^Y$ to be 
$$ \int_{[Q_{h,n}(Y, d)]^{vir}} 
\bbr \cdot
{ev_1^c}^*(\gamma_1) \cup \psi_1^{k_1} \cup \prod_{2 \leq i \leq n}   ev_i^*(\gamma_i) \cup \psi_i^{k_i}, $$
where ${ev_1^c}^*$ is the pullback in compactly supported cohomology. 
Given $\gamma_1, \gamma_2 \in H^*_{CR, ct}(Y)$ and $\gamma_3, \ldots, \gamma_n \in H^*_{CR}(Y)$, define $\br{ \gamma_1 \psi^{k_1}, \ldots, \gamma_n \psi^{k_n}}_{h,n,d}^Y$ to be
$$ \int_{[Q_{h,n}(Y, d)]^{vir}} 
\bbr \cdot
{ev_1}^*(\gamma_1) \cup_{cs} {ev_2}^*(\gamma_2) \cup \psi_1^{k_1} \cup \psi_2^{k_2} \cup \prod_{3 \leq i \leq n}   ev_i^*(\gamma_i) \cup \psi_i^{k_i}. $$
\end{definition}

\begin{definition}
An element $d \in \hom_\ZZ(\widehat G, \QQ)$ is $(V, G, \theta)$-effective if it arises as a finite sum of elements $d_i$ for which there exists a stable quasimap of degree $d_i$.  The set of effective elements forms a semigroup, denoted
by $\eff(V, G, \theta).$
\end{definition}

\subsection{A virtual class for GLSMs}\label{ss:virtGLSM}

Given a GLSM $(V, G, \theta, w)$, the moduli stack of $\epsilon$-stable genus $h$, $n$-pointed, degree $d$ Landau--Ginzburg maps, $LG_{h,n}^{\epsilon}(Y, d)$ was defined in \cite{FJR2}.  In this paper we will focus on the $\epsilon = 0^+$ stability condition.  The moduli space $LG_{h,n}^{0^+}(Y, d)$ will be denoted by $QLG_{h,n}(Y, d)$.  We call this the moduli space of LG quasimaps.

Via the exact sequences \eqref{e:xi},
the group $\hom_\ZZ(\hat \Gamma, \QQ)$ is canonically isomorphic to 
$$\hom_\ZZ(\widehat G, \QQ) \times \hom_\ZZ(\hat \CC^*_R, \QQ).$$ 

As in Definition~\ref{d:frakG}, a principal $\Gamma$-bundle $\cP \to \cC$ naturally defines an element  of $\hom_\ZZ(\hat \Gamma, \QQ)$.  Unless otherwise specified, the \emph{degree} of a principal $\Gamma$-bundle will refer the image of this element under the  projection map $\hom_\ZZ(\hat \Gamma, \QQ) \to \hom_\ZZ(\widehat G, \QQ)$.

\begin{definition}\label{frakGamma}

For $c \in \ZZ$, $h, n \in \ZZ_{\geq 0}$ and $d \in \hom_\ZZ(\widehat G, \QQ)$, let   $\f M_{h,n}(B\Gamma, d)_{\omlog^c}$ denote the moduli stack of genus $h$ $n$-pointed prestable orbi-curves $\cC$ with a degree $d$ principal $\Gamma$-bundle $\cP \to \cC$, and an isomorphism $$\eta\colon \xi_*(\cP) \cong (\omlog^{\otimes c})^\circ,$$ such that the induced morphism $[\cP]\colon \cC \to B\Gamma$ is representable.

Let $\nu \in \widehat \Gamma$ be a \emph{lift} of $\theta$, that is, $\nu|_G = \theta$.  
Denote by $\f M_{h,n}(B\Gamma, d)^\theta_{\omlog^c}$ the open substack of $\f M_{h,n}(B\Gamma, d)_{\omlog^c}$ 
consisting of pairs $(\cC, \cP)$ such that $\cC$ has no rational tails and,  on each irreducible component $\cC '$ of $\cC$ for which $\omlog|_{\cC '}$ is trivial,  the degree 
of $$\cc L_\nu := \cP \times_\Gamma \CC_\nu$$ restricted to $\cC '$ is positive.
We note that for two choices $\nu, \nu '$ of lift, the line bundles $\cc L_\nu$ and $\cc L_{\nu '}$ differ by a power of $\omlog$.  Therefore the definition of $\f M_{h,n}(B\Gamma, d)^\theta_{\omlog^c}$ is \emph{independent} of the lift of $\theta$ \cite[Corollary~4.2.15]{FJR2}.

Let $\pi\colon \f C \to \f M_{h,n}(B\Gamma, d)^\theta_{\omlog^c}$ denote the universal curve and $\cP_\Gamma \to \f C$ the universal principal $\Gamma$-bundle.  Let $\cV_\Gamma := \cP_\Gamma \times_\Gamma V$.
\end{definition}

The construction of the moduli space  $QLG_{h,n}(Y, d)$ and the virtual class is identical to Section~\ref{s:qmap} after replacing $\f M_{h,n}(BG, d)^\theta$ with $\f M_{h,n}(B\Gamma, d)^\theta_{\omlog}$.  As in Definition~\ref{d:sections}, define $\op{tot}(\pi_*\cV_\Gamma)$, this time over $\f M_{h,n}(B\Gamma, d)^\theta_{\omlog}$.  Closed points of $\op{tot}(\pi_*\cV_\Gamma)$ consist of a marked curve $\cC$ together with a section $u_\Gamma \in \Gamma(\cC, \cV_\Gamma|_\cC)$, which defines a map $[u_\Gamma]\colon \cC \to [V/\Gamma]$ to the stack quotient.  
The stack $QLG_{h,n}(Y, d)$ is the open substack of $\op{tot}(\pi_*\cV_\Gamma)$ over $\f M_{h,n}(B\Gamma, d)^\theta_{\omlog}$ satisfying the same condition that  the basepoint locus $[u_\Gamma]^{-1}( [V^{us}(\theta)/\Gamma])$ is a finite set disjoint from nodes and markings.  
%Going forward, let $u_\Gamma \colon \cC \to \cV_\Gamma$ denote the universal section over the universal curve $\cC \to QLG_{h,n}(Y, d)$ and let $[u_\Gamma]\colon \cC \to [V/\Gamma]$ denote the corresponding map.

As described in \cite[Section~4.4]{FJR2}, there exist evaluation maps $ev_i\colon QLG_{h,n}(Y, d) \to IY$ which are defined as follows.  
Let $\Sigma_i \to QLG_{h,n}(Y, d)$ denote the $i$th marked point gerbe (recall that by definition of $ \f M_{h,n}(B\Gamma, d)^\theta_{\omlog}$ this is a trivial gerbe).  
The restriction $[u_\Gamma|_{\Sigma_i}]\colon \Sigma_i \to [V/ \Gamma]$ determines a map $QLG_{h,n}(Y, d) \to I[V/\Gamma]$.  In fact this map factors through $[V/G]$ as follows: The isomorphism
$\eta\colon \xi_*(\cP) \cong (\omlog)^\circ$ induces a map
$$ \cP|_{\Sigma_i} \to 
  (\omlog)^\circ|_{\Sigma_i}  \xrightarrow{=} \cc O_{QLG_{h,n}(Y, d)}^\circ|_{\Sigma_i},$$
  where the second map is the residue at $\Sigma_i$.
The kernel ${\cP}_i '$ gives a principal $G$-bundle on $\Sigma_i$.  There is an isomorphism ${\cP}_i ' \times_G \Gamma \cong \cP|_{\Sigma_i} $, from which it follows that ${\cP}_i ' \times_G V = \cP|_{\Sigma_i} \times_\Gamma V$.  Therefore the restriction $u|_{\Sigma_i} $ defines a section of ${\cP}_i ' \times_G V$, which in turn defines a map $\Sigma_i \to [V/G]$.  Due to the condition that basepoints are disjoint from markings,  in fact $\Sigma_i $ maps to $Y= [V\sslash_\theta G]$, which allows us to define evaluation maps 
$$ev_i\colon QLG_{h,n}(Y, d) \to IY.$$

As before, the relative  obstruction theory over $\f M_{h,n}(B\Gamma, d)^\theta_{\omlog}$ is given by $\RR \pi_*(\cV_\Gamma)^\vee.$
Proposition~\ref{p:res1} holds over $\f M_{h,n}(B\Gamma, d)^\theta_{\omlog}$ as well.  Thus we can define a smooth Deligne--Mumford stack $U \to \f M_{h,n}(B\Gamma, d)^\theta_{\omlog}$, a vector bundle $E \to U$, and a natural section $\beta \in \Gamma(U, E)$ such that 
$QLG_{h,n}(Y, d) = \{\beta = 0\}$.  As in the previous section, we make the following definition.
\begin{definition}\label{d:virt} Define the virtual class $[QLG_{h,n}(Y, d)]^{vir}$ to be
$$[QLG_{h,n}(Y, d)]^{vir} := e_{loc}(E, \beta) = \op{cl}\left(0^!( [U])\right),$$
where $0^!: A_*(U) \to A_*(QLG_{h,n}(Y, d))$ is the Gysin pullback via the diagram
\[
\begin{tikzcd}
QLG_{h,n}(Y, d) \ar[r] \ar[d] & U \ar[d, "\beta"] \\
U \ar[r, "0"] & \tot(E).\end{tikzcd}
\]
\end{definition}

In general, the virtual class of Definition~\ref{d:virt} is not sufficient for defining enumerative GLSM invariants because the moduli space $QLG_{h,n}(Y, d)$ is not usually proper, nor are the evaluation maps.  Furthermore, this virtual class does not take into account the potential $w\colon Y \to \CC$.

\subsection{GLSM invariants via cosection localization}\label{ss:cosloc}

In this section we recall the definition of compact type GLSM invariants proposed by Fan--Jarvis--Ruan in \cite{FJR2, FJR2b}.\footnote{The construction of the GLSM invariants appearing in the most recent preprint version of the paper \cite[Section~6.1]{FJR2b} is more detailed and slightly more general than that appearing in the published version \cite{FJR2}.  Throughout this section we will follow the construction described in the preprint \cite{FJR2b}.}  For compact type insertions $\gamma_1, \ldots, \gamma_n \in \cc H_{ct}(Y, w)$ satisfying Assumption~\ref{a:FJR} below, they define the GLSM invariant $$\br{ \gamma_1 \psi^{k_1}, \ldots, \gamma_n \psi^{k_n}}_{h,n,d}^{(Y, w)}$$ in terms of a cosection localized virtual class.    This section serves as motivation for our eventual definition of  compact type GLSM invariants in the case of proper evaluation maps (Definition~\ref{d:propevinvts}), however it is not strictly necessary for the rest of the paper.  

Let $\gamma_1, \ldots, \gamma_n \in \cc H_{ct}(Y, w)$ be compact type insertions.
By definition of $\cc H_{ct}(Y, w)$, $\gamma_i$  may be expressed as  ${j'_i}_*(\beta_i)$ for $j_i\colon Z_i \hookrightarrow IY$ an inclusion of a smooth proper substack into the locus $\{w = 0\}$ and $j'_i\colon (Z_i,0) \to (IY, w)$ the map of LG models.  Without loss of generality we may assume that $Z_i$ is contained in a single twisted sector $Y_{g_i} = [(V^{g_i})^{ss} /  Z_G(g_i)]$.  Following \cite[Section~6.1]{FJR2b}, we assume for $1 \leq i \leq n$ that:
\begin{assumption}\label{a:FJR}
There exists a $\Gamma$-invariant subspace $H_i \subset V^{g_i}$ such that $Z_i = [H_i^{ss} / Z_G(g_i)]$.
\end{assumption}

Let $$\underline{ev}:  QLG_{h,n}(Y, d) \to IY \times \cdots \times IY $$ denote the product of the evaluation maps.  Denote by $\underline Z$ the product 
$Z_1  \times \cdots \times Z_n \subset IY \times \cdots \times IY$.
Define the substacks  $$QLG_{\underline Z}(Y, d) : = \underline{ev}^{-1}(\underline Z)$$
and 
$$QLG_{\underline Z}(Crit(w), d) = QLG_{\underline Z}(Y, d) \cap QLG_{h,n}(Crit(w), d),$$
where $QLG_{h,n}(Crit(w), d)$ denotes the locus of LG quasimaps to the critical locus $[\{dw = 0\} /\Gamma] \subset [V/\Gamma]$.
As explained in \cite{FJR2b} there is a modified perfect obstruction theory on $QLG_{\underline Z}(Y, d)$, which we recall below.  

First, we restrict from $\f M_{h,n}(B\Gamma, d)^\theta_{\omlog}$ to $\f M_{h,(\underline g)}(B\Gamma, d)^\theta_{\omlog}$, the open and closed substack defined by the condition that the generator of the isotropy at the $i$th marked point maps to $g_i \in G$.  Let $\pi_{\underline g}: \f C_{\underline g} \to \f M_{h,(\underline g)}(B\Gamma, d)^\theta_{\omlog}$ denote the universal curve.  
On $\f C_{\underline g}$, consider the restriction map $\cV_\Gamma \to \cV_\Gamma|_{\Sigma_i} = {\sigma_i}_* \sigma_i^* \cV_\Gamma$.  The vector bundle $\sigma_i^* \cV_\Gamma$ splits into eigen-bundles with respect to the action of $g_i$.  Thus there is a canonical projection $\sigma_i^* \cV_\Gamma \to (\sigma_i^* \cV_\Gamma)^{g_i}$.
Define 
$$\cV_\Gamma|_{\Sigma_i}^{g_i}:={\sigma_i}_* (\sigma_i^* \cV_\Gamma)^{g_i}.$$

Let $\cH_i$ denote the vector bundle $\cP_\Gamma \times_\Gamma H_i$ on $\f C_{\underline g}$.  Let $\cH_i^\perp$ denote the quotient $(\sigma_i^* \cV_\Gamma)^{g_i}/ \sigma_i^* H_i$. By composing the map $\cV_\Gamma \to \cV_\Gamma|_{\Sigma_i}$ with 
\begin{equation}\label{e:perpres}
{\sigma_i}_* \sigma_i^* \cV_\Gamma \to {\sigma_i}_* (\sigma_i^* \cV_\Gamma)^{g_i} \to  {\sigma_i}_* \cH_i^\perp
\end{equation} for each $i$, we obtain the following surjective morphism of sheaves
\begin{equation}\label{e:surj} \cV_\Gamma \to \bigoplus_{i=1}^n  {\sigma_i}_* \cH_i^\perp.\end{equation}
Denote the kernel by $\widetilde{\cV}_\Gamma$.  A local computation shows that $\widetilde{\cV}_\Gamma$ is a vector bundle on $\f C_{\underline g}$ (see the discussion following Definition 6.1.2. in \cite{FJR2b}).

Observe that $QLG_{\underline Z}(Y, d) $ is the locus  in $QLG_{h,(\underline g)}(Y, d)$ for which the section of $\cV_\Gamma$ vanishes outside of $ \cH_i$ after restricting to $\Sigma_i$.  In other words, $QLG_{\underline Z}(Y, d) $ is defined as the intersection of $$QLG_{h,(\underline g)}(Y, d) := \underline{ ev}^{-1}(Y_{g_1} \times \cdots \times Y_{g_n})$$ with 
$$\tot \pi_*( \widetilde{\cV}_\Gamma)  \subset \tot \pi_* (\cV_\Gamma).$$

On $QLG_{\underline Z}(Y, d) $, there is a relative perfect obstruction theory given by $\RR \pi_*(\widetilde \cV_\Gamma)^\vee.$  By \cite[Lemma~6.1.5]{FJR2b},  on the corresponding absolute obstruction theory there exists a cosection
$$\text{Obs}_{QLG_{\underline Z}(Y, d)} \to \mathcal{O}_{QLG_{\underline Z}(Y, d)}.$$ The associated cosection localized virtual class \cite{KL} is supported on the proper stack $QLG_{\underline Z}(Crit(w), d)$.  We denote the cosection localized virtual class by 
$$[QLG_{\underline Z}(Crit(w), d)]^{vir} := [QLG_{\underline Z}(Y, d)]^{vir}_{loc}.$$
With this setup one can make the following definition:
\begin{definition}\label{d:ctinvts}
Let $\gamma_1, \ldots, \gamma_n$ be classes arising as the pushforward of $\beta_1, \ldots, \beta_n$ via maps $H^*(Z_i) \to \cc H_{ct}(Y, w)$.  Under Assumption~\ref{a:FJR},
define the GLSM invariant $\br{ \gamma_1 \psi^{k_1}, \ldots, \gamma_n \psi^{k_n}}_{h,n,d}^{Y,w}$ as
\begin{equation}\label{e:FJRdef} 
\int_{[QLG_{\underline Z}(Crit(w), d)]^{vir}} \bbr \cdot \prod_{i=1}^n \left(ev_i^*(\beta_i) \cup \psi_i^{k_i} \right).
\end{equation}
\end{definition}
Consider the following diagram
\begin{equation}\label{e:SZ}
\begin{tikzcd}
  U_{\underline Z} \ar[d, "\underline{\wt{ev}}_Z"] \ar[r, "\wt j"] &U \ar[d, "\underline{\wt{ev}}"]\\
Z_1  \times \cdots \times Z_n \ar[r, "j"] &IY \times \cdots \times IY  
\end{tikzcd}
\end{equation}
where $U_{\ul Z}$ is defined as the fiber product.  Because the map $A \to \pi_*( \cV_\Gamma|_\Sigma)$ is surjective, $\underline{\wt{ev}}$ is smooth.
We have the following relationship between $ [QLG_{\underline Z}(Crit(w), d)]^{vir}$ and the virtual class $[QLG_{h,n}(Y, d)]^{vir}$ of Definition~\ref{d:virt}.
\begin{proposition}\label{p:cosvir}
Let  $c\colon QLG_{\underline Z}(Crit(w), d) \to QLG_{\underline Z}(Y, d)$ denote the inclusion.  Then
$$ c_*\left([QLG_{\underline Z}(Crit(w), d)]^{vir}\right) = e_{loc}(\wt j^* E, \wt j^*\beta)$$
in $H_*(QLG_{\underline Z}(Y, d))$.
\end{proposition}

\begin{proof}
By \cite[Theorem 5.1]{KL}, the pushforward of the cosection localized virtual class $c_*\left([QLG_{\underline Z}(Crit(w), d)]^{vir}\right)$ is simply the usual virtual class $[QLG_{\underline Z}(Y, d)]^{vir}$ on $QLG_{\underline Z}(Y, d)$, with respect to the relative perfect obstruction theory $\RR \pi_*(\widetilde \cV_\Gamma)^\vee.$  Thus it suffices to show that $[QLG_{\underline Z}(Y, d)]^{vir} = e_{loc}(\wt j^* E,\wt j^* \beta)$.

Recall the discussion preceding Definition~\ref{d:virt}.  Let $\cc A \to \cc B$ be a resolution of $\cV_\Gamma|_{\f C_{\underline g}}$ satisfying the properties of 
Proposition~\ref{p:res1}.  Compose the surjective map $\cc A \to \cV_\Gamma|_{\Sigma_i}$ with \eqref{e:perpres} and let $\wt{\cc A}$ denote the kernel.  We have the following commutative diagram of sheaves on $\f C_{\underline g}$, where all columns and rows are short exact sequences.
\begin{equation}\label{e:square0}
\begin{tikzcd}[cramped]
& 0 \ar[d]&0 \ar[d]&0 \ar[d]&\\
0\ar[r]& \wt \cV_\Gamma \ar[d] \ar[r]&\cV_\Gamma \ar[d] \ar[r]& \bigoplus_{i=1}^n  {\sigma_i}_* \cH_i^\perp \ar[d, "id"] \ar[r] &0\\
0\ar[r]& \wt{\cc A} \ar[d] \ar[r]& \cc A \ar[d] \ar[r]&\bigoplus_{i=1}^n  {\sigma_i}_* \cH_i^\perp \ar[d] \ar[r] &0\\
0\ar[r] & \cc B \ar[d] \ar[r, "id"]& \cc B \ar[r] \ar[d] & 0  & \\
& 0 &0 & &
\end{tikzcd}
\end{equation}
It follows from the second bullet in Proposition~\ref{p:res1} that the map $${\pi_{\underline g}}_*(\cc A) \to {\pi_{\underline g}}_*\left( \bigoplus_{i=1}^n  {\sigma_i}_* \cH_i^\perp\right)$$ is surjective, therefore $\RR^1 {\pi_{\underline g}}_*(\wt {\cc A}) = 0$.  Let $\wt A$ denote the vector bundle ${\pi_{\underline g}}_*(\wt {\cc A})$.

Pushing forward \eqref{e:square0}, we observe that $U_{\underline Z}$ is the intersection of $U$ with $\tot ( \wt A) \subset \tot (A)$.  Furthermore, the two-term complex 
$$[\wt j^* E^\vee \xrightarrow{\beta^\vee} \wt j^*  p_A^*(\wt A)^\vee]$$ gives a resolution of $\RR \pi_*(\widetilde \cV_\Gamma)^\vee$ on $U_{\underline Z}$, and $QLG_{\underline Z}(Y, d)$ is the zero locus of the section $\wt j^* \beta \in \Gamma( U_{\underline Z}, \wt j^* E)$.  By the general discussion in \cite[Section~6]{BF}, it follows that $[QLG_{\underline Z}(Y, d)]^{vir}$ is the localized top Chern class $e_{loc}(\wt j^* E, \wt j^* \beta)$ as desired.
\end{proof}

\subsection{The special case of proper evaluation maps}

In this section we explain how, if the evaluation maps 
$ ev_i\colon  QLG_{h,n}(Y, d) \to IY$
are proper,  there is a simple expression for compact type GLSM invariants that does not require use of a cosection.

 \begin{proposition}\label{p:propevinvts} 
 Let $\gamma_1, \ldots, \gamma_n$ be classes arising as the pushforward of $\beta_1, \ldots, \beta_n$ via maps 
$H^*(Z_i) \to \cc H_{ct}(Y, w)$ as in Section~\ref{ss:cosloc}, with each $Z_i$ satisfying Assumption~\ref{a:FJR}.   Assume that $ev_i\colon QLG_{h,n}(Y, d) \to IY$ is proper for $1 \leq i \leq n$.  Then \eqref{e:FJRdef} 
 is equal to
 \begin{equation}\label{e:propevdef} 
\int_{[QLG_{h,n}(Y, d)]^{vir}} \bbr \cdot {ev_1^c}^*({j_1}^c_* \beta_1) \cup \cdots \cup  {ev_n^c}^*({j_n}^c_* \beta_n) \cup \prod_{i=1}^n \left( \psi_i^{k_i} \right),
\end{equation}
where ${ev_i^c}^*$ and ${j_i^c}_*$ denote the pullback and pushforward in compactly supported cohomology, and $[QLG_{h,n}(Y, d)]^{vir}$ is the (non-localized) virtual class of Definition~\ref{d:virt}.
\end{proposition}
 
 \begin{proof}
 
 By assumption,
 $$QLG_{\underline Z}(Y, d) = U_{\underline Z} \cap QLG_{h,n}(Y, d) = \underline{ev}^{-1}(\underline Z)$$
is proper.  

By Proposition~\ref{p:cosvir}, equation \eqref{e:FJRdef} 
 is equal to 
\begin{align}\label{e:propev2} 
&\int_{QLG_{\underline Z}(Y, d)} \bbr \cdot \prod_{i=1}^n ev_i^*(\beta_i) \cup \prod_{i=1}^n \left( \psi_i^{k_i} \right)\cap e_{loc}(\wt j^* E,\wt j^* \beta)
\end{align}
which simplifies to
\begin{align} \label{e:propev3}
&\int_{QLG_{\underline Z}(Y, d)} \bbr \cdot \prod_{i=1}^n ev_i^*(\beta_i) \cup \prod_{i=1}^n \left( \psi_i^{k_i} \right)\cap \wt j^! (e_{loc}( E, \beta)) \\ \nonumber
=&\int_{QLG_{\underline Z}(Y, d)} \bbr \cdot \prod_{i=1}^n ev_i^*(\beta_i) \cup \prod_{i=1}^n \left( \psi_i^{k_i} \right)\cap j^! (e_{loc}( E, \beta)) \\ \nonumber
=&\int_{\underline Z} \bbr \cdot\underline{ev}_*\left({ \underline{ev}}^*(\underline{\beta})  \cup \prod_{i=1}^n \left( \psi_i^{k_i} \right)\cap j^! (e_{loc}( E, \beta))\right)
\\ \nonumber
=&\int_{\underline Z} \bbr \cdot \underline{\beta} \cap \underline{ev}_* j^! \left(  \prod_{i=1}^n \left( \psi_i^{k_i} \right) \cap e_{loc}( E, \beta)\right)
\\ \nonumber
=&\int_{\underline Z} \bbr \cdot \underline{\beta} \cap j^* {\underline{ev}}_* \left( \prod_{i=1}^n \left( \psi_i^{k_i} \right) \cap e_{loc}( E, \beta)\right).
\end{align}
The  equality of \eqref{e:propev2}  with the first line is by definition of $e_{loc}(E, \beta)$  and \cite[Proposition~14.1 (d)(ii)]{Fulton}.  The first equality is compatibility of Gysin pullbacks, \cite[Theorem~6.2 (c)]{Fulton}.  The final equality is \cite[Theorem~6.2 (a)]{Fulton}.  Here $\underline{\beta}$ denotes the exterior product of  $\beta_1$ through $\beta_n$.

Let $j_*^c\colon H^*(\underline Z) \to H^*_{cs}(\prod_{i=1}^n I Y)$ denote the pushforward to compactly supported cohomology.  Then \eqref{e:propev3} is equal to
\begin{align}\label{e:propev4} 
&\int_{\prod_{i=1}^n I Y}  j_*^c(\underline{\beta}) \cap {\underline{ev}}_* \left(\bbr \cdot  \prod_{i=1}^n \left( \psi_i^{k_i} \right) \cap e_{loc}( E, \beta)\right)
\\ \nonumber
=&\int_{QLG_{h,n}(Y, d)} \bbr \cdot {\underline{ev}^c}^*j_*^c(\underline{\beta})  \cup \prod_{i=1}^n \left( \psi_i^{k_i} \right)\cap  e_{loc}( E, \beta)
\\ \nonumber
=&\int_{QLG_{h,n}(Y, d)} \bbr \cdot {\underline{ev}^c}^*\left(\otimes_{i=1}^n {j_i}_*^c({\beta_i})  \right) \cup \prod_{i=1}^n \left( \psi_i^{k_i} \right)\cap  e_{loc}( E, \beta)
\\ \nonumber
=&\int_{QLG_{h,n}(Y, d)} \bbr \cdot {ev_1^c}^*({j_1}^c_* \beta_1) \cup \cdots \cup  {ev_n^c}^*({j_n}^c_* \beta_n) \cup \prod_{i=1}^n \left( \psi_i^{k_i} \right)\cap  e_{loc}( E, \beta)
\end{align}
The first and second line are each applications of the projection formula (with respect to the maps $j$ and $\underline{ev}$ respectively).  The second equality is simply by functoriality of the Kunneth formula, which holds for compactly supported cohomology as well.

The last line of \eqref{e:propev4}  is exactly \eqref{e:propevdef}, finishing the proof.
\end{proof}

\begin{lemma}\label{l:ccs}
If $n\geq 2$, the expression $ {ev_1^c}^*({j_1}^c_* \beta_1) \cup \cdots \cup  {ev_n^c}^*({j_n}^c_* \beta_n)$ in Proposition~\ref{p:propevinvts} is equal to
$ev_1^*(\phi^w(\gamma_1)) \cup_{cs} ev_2^*(\phi^w( \gamma_2)) \cup \cdots \cup  ev_n^*(\phi^w( \gamma_n))$.
%\footnote{Here $ev_1^*({j_1}_* \beta_1) \cup_{cs} \cdots \cup_{cs}  ev_n^*({j_n}_* \beta_n)$ lies in compactly supported cohomology $H^*_{cs}(QLG_{h,n}(Y, d))$, so the cap product with $[QLG_{h,n}(Y, d)]^{vir}$ lies in homology of chains with finite support (as opposed to Borel--Moore homology).  Thus the integral (which is pushforward to a point and taking degree) is well-defined.} \note{probably want to remove this footnote, at any rate, its in the wrong spot}
\end{lemma}
\begin{proof}
As shown in \cite{Shoe1}, for a proper map $f$, $\phi^{cs} \circ {f^c}^* = f^* \circ \phi^{cs}$, and $\phi^{cs} \circ f^c_* = f_* \circ \phi^{cs}$.
%Because $ev_i$ and $j_i$ are both proper, they commute with the map $\phi^{cs}$ from compactly supported cohomology to cohomology.  
Then under our assumption that $n \geq 2$, by \eqref{e:prodcomp} we have 
\begin{align*}
&{ev_1^c}^*({j_1}^c_* \beta_1)  \cup  {ev_2^c}^*({j_2}^c_* \beta_2) \\
=&\phi^{cs}({ev_1^c}^*({j_1}^c_* \beta_1)) \cup_{cs}   \phi^{cs}({ev_2^c}^*({j_2}^c_* \beta_2)) \\
=& ev_1^*({j_1}_*\phi^{cs}( \beta_1)) \cup_{cs}   ev_2^*({j_2}_*\phi^{cs}( \beta_2))\\
=& ev_1^*({j_1}_*( \beta_1)) \cup_{cs}   ev_2^*({j_2}_*( \beta_2))\\
=& ev_1^*(\phi^w(\gamma_1)) \cup_{cs}   ev_2^*(\phi^w( \gamma_2))
\end{align*}
\end{proof}

Note that in the presence of proper evaluation maps,
 \eqref{e:propevdef} is defined regardless of Assumption~\ref{a:FJR}.
 Motivated then by Proposition~\ref{p:propevinvts}, we use \eqref{e:propevdef} and Lemma~\ref{l:ccs}
 to define compact type GLSM invariants in the case that the evaluation maps are proper.   

\begin{definition}\label{d:propevinvts} Let $\gamma_1, \ldots, \gamma_n$ be classes in $\cc H_{ct}(Y, w)$.
Assume that $n\geq 2$ and the evaluation maps $ev_i$ are proper for $1 \leq i \leq 2$.  Define the GLSM invariant $\br{ \gamma_1 \psi^{k_1}, \ldots, \gamma_n \psi^{k_n}}_{h,n,d}^{Y,w}$ as:
 \begin{equation}\label{e:propevdef2} 
\int_{ [QLG_{h,n}(Y, d)]^{vir}}  \bbr \cdot ev_1^*(\phi^w(\gamma_1)) \cup_{cs} ev_2^*(\phi^w( \gamma_2)) \cup \cdots \cup  ev_n^*(\phi^w( \gamma_n)) \cup \prod_{i=1}^n \left( \psi_i^{k_i} \right).
\end{equation}
 \end{definition}

\begin{remark}
We note that Definition~\ref{d:ctinvts} and \eqref{e:propevdef} depend, a-priori,  on the choices of $\beta_1, \ldots, \beta_n$ which are pushed forward via ${j_i'}_*$ to obtain $\gamma_1, \ldots, \gamma_n$.    Definition~\ref{d:propevinvts} on the other hand does not rely on such choices.
\end{remark}

Elements of the kernel of the map 
$$\phi^w\colon \cc H(Y, w) \to H^*_{CR}(Y)$$
are often referred to as \emph{broad} sectors.  There are conjectures (see, e.g. \cite[Conjecture 2.13]{CJR}) about when GLSM invariants with broad insertions are zero.  Such results are known as \emph{broad vanishing}.  The following ``corollary'' is immediate from our definition.
\begin{corollary}\label{c:bv}
Broad vanishing holds for compact type insertions in the presence of proper evaluation maps.  More precisely, 
assume that $ev_i\colon QLG_{h,n}(Y, d) \to IY$ is proper for $1 \leq i \leq n$, and that $n \geq 2$.
For $\gamma_1, \ldots, \gamma_n \in \cc H_{ct}(Y, w)$, the invariant $\br{ \gamma_1 \psi^{k_1}, \ldots, \gamma_n \psi^{k_n}}_{h,n,d}^{Y,w}$ is zero if any $\gamma_i$ lies in the broad subspace $\ker(\phi^w) \subset \cc H(Y, w)$.
\end{corollary}

In practice Definition~\ref{d:propevinvts} is quite restrictive.  Evaluation maps from moduli spaces of LG maps are rarely proper.  In the next section we show that in the special case that $h=0$ and $n=2$, the evaluation maps are always proper, and the definition applies.  We will see in later sections that this case is of special relevance to mirror symmetry.

\section{Genus zero, two marked points}\label{s:g02mp}

Let $\pi\colon \f C \to \f M_{0,2}(BG, d)^\theta$ denote the universal curve.  The following observation is immediate from the definitions. 
\begin{lemma}[{\cite[Lemma~3.1]{HS}}] \label{l:chain} 
Let $\spec \CC \to  \f M_{0,2}(BG, d)^\theta$ be a morphism and let $\cC$ be the curve obtained by pulling back $\pi$.
The coarse underlying curve $\underline{\cC}$ is  a chain of $\PP^1$'s with a marked point at either end of the chain.
The same holds for $\f M_{0,2}(B\Gamma, d)_{\omega_{\pi, \op{log}}}^\theta$.
\end{lemma}
Lemma~\ref{l:chain} implies the following simple but important fact.
\begin{lemma}\label{l:cantriv}
Over both  $\f M_{0,2}(BG, d)^\theta$ and $ \f M_{0,2}(B\Gamma, d)_{\omlog}^\theta$, the log-canonical bundle $\omega_{\pi, \op{log}} = \omega_\pi(\Sigma_1 + \Sigma_2)$ is canonically trivialized by a nowhere-vanishing global section $\cc O_{\f C} \to \omega_{\pi, \op{log}}$ characterized by the property that on any fiber the residue around $p_1$ is equal to $2\pi i$.
\end{lemma}
\begin{proof}
  We will prove the statement for $ \f M_{0,2}(B\Gamma, d)_{\omlog}^\theta$ but the proof is identical for $\f M_{0,2}(BG, d)^\theta$.

Consider the long exact sequence
\begin{align}\label{e:LESsec} 0 \to &\pi_* (\omega_\pi(\Sigma_2)) \to  \pi_* (\omega_\pi(\Sigma_1 + \Sigma_2)) \xrightarrow{\rho} \omega_\pi(\Sigma_1 + \Sigma_2)|_{\Sigma_1} \\ \to  R^1&\pi_* (\omega_\pi(\Sigma_2)) \to R^1\pi_* (\omega_\pi(\Sigma_1 + \Sigma_2)) \to 0. \nonumber \end{align}
By degree considerations and an induction argument, $H^0(\omega_{\cC}(p_2)) = 0$ for any source curve $\cC$.
By Serre duality $H^1 (\omega_{\cC}(p_2))$ is isomorphic to  $H^0 (\cc O_{\cC}(-p_2))$ which is also zero for any source curve $\cC$.  Therefore $$R^0\pi_* (\omega_\pi(\Sigma_2)) = R^1\pi_* (\omega_\pi(\Sigma_2)) = 0$$ and the map $\rho$ is an isomorphism.  Furthermore, the residue map provides a canonical isomorphism 
\[\op{Res}: \omega_\pi(\Sigma_1 + \Sigma_2)|_{\Sigma_1} \xrightarrow{=} \cc O_{ \f M_{0,2}(B\Gamma, d)_{\omlog}^\theta}.\]
The constant section corresponding to $2 \pi i \in \Gamma(  \f M_{0,2}(B\Gamma, d)_{\omlog}^\theta, \cc O_{ \f M_{0,2}(B\Gamma, d)_{\omlog}^\theta})$ is therefore the image of a global section 
\[s \in \Gamma(  \f M_{0,2}(B\Gamma, d)_{\cc O_{\f C}}^\theta, \pi_*(\omega_\pi(\Sigma_1 + \Sigma_2))).\]
This section is itself defined by a global map $f\colon \cc O_{\cC} \to \omega_\pi(\Sigma_1 + \Sigma_2)$.  

We will show that $f$ is nowhere-vanishing, which can be checked fiberwise.  Let $\spec \CC \to   \f M_{0,2}(B\Gamma, d)_{\omlog}^\theta$ be a morphism and let $\cC$ be the corresponding curve.  
By Lemma~\ref{l:chain}, the log canonical bundle $\omega_{\cC}(p_1 + p_2)$ is degree zero on each irreducible component, and is therefore trivial.
By considering the composition $\op{Res} \circ \rho$, we see that locally $f$ sends $1$ to $dx/x$, where $x$ is a local coordinate around $p_1$.  
This local section extends uniquely to a section of $\omega_{\cC}(p_1 + p_2)$, which is nowhere vanishing because $\omega_{\cC}(p_1 + p_2)$ is trivial.
\end{proof}
\begin{proposition}\label{p:stackeq}
There is a canonical isomorphism $\f M_{0,2}(BG, d)^\theta \cong \f M_{0,2}(B\Gamma, d)_{\omlog}^\theta$.  Under this isomorphism, the universal curves are identified.  The principal bundles are related via
\begin{equation}\label{e:rel1} \cP_\Gamma = \cP_G \times_{\langle J \rangle} \CC^*\end{equation}
and fit into an exact sequence
\begin{equation}\label{e:rel2} 0 \to \cP_G \to \cP_\Gamma \to \xi_*(\cP_\Gamma) \cong \omega_{\pi, \op{log}}^\circ \cong \cc O_{\f C}^\circ \to 0.\end{equation}
\end{proposition}

\begin{proof}
By Lemma~\ref{l:chain}, any fiber of the universal curve $\f C \to \f M_{0,2}(B\Gamma, d)_{\omlog}^\theta$ consists of a chain of irreducible components.  By the previous proposition, there is a canonical isomorphism $ \cc O_{\f C} \to \omega_{\pi, \op{log}}  $  over $\f M_{0,2}(B\Gamma, d)_{\omlog}^\theta$.  We can therefore identify $\f M_{0,2}(B\Gamma, d)_{\omlog}^\theta$ with $\f M_{0,2}(B\Gamma, d)_{\cc O_{\cC}}^\theta$.  

Therefore, it suffices to find a canonical isomorphism 
$i_1\colon \f M_{0,2}(BG, d)^\theta \to \f M_{0,2}(B\Gamma, d)_{\cc O_{\cC}}^\theta$.
The morphism is defined as follows.  For a family of curves $\cC \to S$ and a principal $G$-bundle $\cP_G \to \cC$, we define the principal $\Gamma$-bundle 
$$\cP_\Gamma := \cP_G \times_{\langle J \rangle} \CC^*,$$
where $J$ acts on $\CC^*$ by $e^{2 \pi i/d_w}$.  Note that there is a canonical isomorphism from
$$\xi_*(\cP_\Gamma) = \cP_\Gamma/G = \CC^*/\langle J \rangle \times \cC$$
to $\cc O_{\cC}^\circ = \CC^*\times \cC$.  This defines the morphism $i_1$.

On the other hand there is a map $i_2\colon \f M_{0,2}(B\Gamma, d)_{\cc O_{\cC}}^\theta \to \f M_{0,2}(BG, d)^\theta$ defined as follows.  For a family of curves $\cC \to S$ and a principal $\Gamma$-bundle $\cP_\Gamma \to \cC$ with an isomorphism $\eta: \cP_\Gamma/G \cong \cc O_{\cC}^\circ$, define $\cP_G$ as the kernel of the map
$$ \cP_\Gamma \to \cP_\Gamma/G \xrightarrow{\eta} \cc O_{\cC}^\circ.$$  It is straightforward to check that $\cP_G \to \cC$ is a principal $G$-bundle.  This defines the morphism $i_2$.

Given a principal $G$-bundle $\cP_G $ arising as the kernel of a map $\cP_\Gamma \to  \cc O_{\cC}^\circ$, there is a well defined map $\cP_G \times \CC^* \to \cP_\Gamma$ given by $(g, \lambda) \mapsto \lambda \cdot i(g)$, where $i: \cP_G \to \cP_\Gamma$ is the inclusion and the action of $\CC^*$ on $\cP_\Gamma$ is defined by identifying the torus $\CC^*$ with $\CC^*_R \subset \Gamma$.  This induces an isomorphism 
$$\cP_G \times_{\langle J \rangle} \CC^* \cong \cP_\Gamma.$$
It follows that the morphisms $i_1$ and $i_2$ are inverse to one another.
\end{proof}
\begin{proposition}\label{p:modint}
There is a canonical isomorphism 
$$Q_{0,2}(Y, d) \cong QLG_{0,2}(Y, d).$$
Under this isomorphism, the virtual  classes $[Q_{0,2}(Y, d)]^{vir}$ and $[QLG_{0,2}(Y, d)]^{vir}$ are identified.
\end{proposition}

\begin{proof}
The space $ Q_{0,2}(Y, d)$ (resp. $QLG_{0,2}(Y, d) $) is an open subset of the space 
of sections of $\cP_G \times_G V$ (resp. $\cP_\Gamma \times_\Gamma V$)
over 
$\f M_{0,2}(BG, d)^\theta $ (resp. $ \f M_{0,2}(B\Gamma, d)_{\omlog}^\theta$).  
By Proposition~\ref{p:stackeq}, $\f M_{0,2}(BG, d)^\theta $ and $ \f M_{0,2}(B\Gamma, d)_{\omlog}^\theta$ are equal.  
On the universal curve over $\f M_{0,2}(BG, d)^\theta = \f M_{0,2}(B\Gamma, d)_{\omlog}^\theta$, the inclusion
$\cP_G \hookrightarrow \cP_\Gamma$ from \eqref{e:rel1}  
allows one to identify the vector bundles 
\begin{equation} \label{e:univbund} \cV_G = \cP_G \times_G V = \cP_\Gamma \times_\Gamma V = \cV_\Gamma.\end{equation}
Both  $ Q_{0,2}(Y, d)$ and $QLG_{0,2}(Y, d) $ are then defined as the subspace of $\tot( \pi_*( \cV_G)) = \tot( \pi_*( \cV_\Gamma))$ where the section $u$ maps to the $\theta$-unstable locus at finitely many points, each distinct from the marked points and nodes.  This proves the first statement.

By \eqref{e:univbund}, the  obstruction theories for $ Q_{0,2}(Y, d)$ and $QLG_{0,2}(Y, d) $ relative to $\f M_{0,2}(BG, d)^\theta = \f M_{0,2}(B\Gamma, d)_{\omlog}^\theta$ are equal:
$$ \RR \pi_*(\cV_G)^\vee =  \RR \pi_*(\cV_\Gamma)^\vee$$
thus the virtual classes coincide.
\end{proof}

\begin{corollary}\label{c:2pt}  If $h=0$ and $n=2$, the evaluation maps $ev_i\colon QLG_{0,2}(Y, d) \to IY$ are proper.
Given classes $\gamma_1, \gamma_2 \in \cc H_{ct}(Y, w)$ and $d \in \op{Eff}([V/G])$, 
$$\br{\gamma_1 \psi^{k_1}, \gamma_2\psi^{k_2}}_{0,2,d}^{(Y, w)} = \br{\phi^w(\gamma_1)\psi^{k_1}, \phi^w(\gamma_2)\psi^{k_2}}_{0,2,d}^{Y}.$$
\end{corollary}
\begin{proof}
Proposition~\ref{p:modint} identifies the moduli spaces
$Q_{0,2}(Y, d)$  and $QLG_{0,2}(Y, d).$   It is left to check that the respective evaluation maps 
are equal under this identification.  This follows from the definition of the evaluation maps on $QLG_{0,2}(Y, d)$ (Section~\ref{ss:virtGLSM}), together with the fact that the universal principal $G$- and $\Gamma$-bundles 
are related via the exact sequence \eqref{e:rel2}.  The first claim then holds by Lemma~\ref{l:propev}.

The equality of invariants is then immediate from Definition~\ref{d:propevinvts} and Proposition~\ref{p:modint}. 
\end{proof}

\section{Adding light points}\label{s:lp}
In this section we recall the definition of quasimap invariants with light points, and define GLSM invariants with light points in a special case.
\subsection{Quasimap invariants}\label{quasimap invariants}
In \cite{CKbig}, Ciocan-Fontanine--Kim define quasimap invariants with light points, and use these to define a new class of generating functions of quasimap invariants.  These functions are shown to be  related to more standard generating functions of Gromov--Witten invariants via wall crossing, but they have the benefit of being readily computable.  

In the language of quasimaps, light points are obtained by replacing the moduli space 
$ Q_{h,n}(Y, d)$ with $ Q_{h,n}(Y \times (\PP^0)^k, (d,1^k))$, where $(d, 1^k)$ is shorthand for $(d, 1, \ldots, 1) \in \hom_\ZZ( \widehat{(G \times (\CC^*)^k)} , \QQ)$.    The base point of a degree one quasimap to $\PP^0:= [\CC\sslash_{\op{id}} \CC^*]$ marks a single point on the source curve.  We call these points light because they may collide with one another in the moduli space of stable quasimaps.  

\begin{definition}
Define the moduli space of genus $h$, $n$-pointed, degree $d$ quasimaps with $k$ light points to be
$$Q_{h,n|k}(Y, d):= Q_{h,n}(Y \times (\PP^0)^k, (d,1^k)).$$
Define $$[Q_{h,n|k}(Y, d)]^{vir}:= [Q_{h,n}(Y \times (\PP^0)^k, (d,1^k))]^{vir}.$$

\end{definition}

Label the light marked points as $q_1, \ldots, q_k$.  Note that light marked points may collide with each other, but not with heavy points or nodes.  On the quasimap moduli space $Q_{h,n|k}(Y, d)$, there exist evaluation maps $\hat {ev}_j$ at the light marked points which map to the stack quotient \cite[Section~2.3]{CKbig}.  They may be defined as the composition 
$$\hat {ev}_j^Q \colon Q_{h,n|k}(Y, d) \xrightarrow{\hat s_j} \cC  \xrightarrow{[u_G]} [V/G],$$
where $\hat s_j\colon Q_{h,n|k}(Y, d) \to \cC $ is the section of the universal curve given by the $j$th light point.
We note that the light markings defined here have no orbifold structure.  

Quasimap invariants with light point insertions are  defined in \cite{CKbig}.  The following is a slight modification of the definition in \cite{CKbig} to the case when $Y$ is not assumed to be proper.
\begin{definition}
Given $\gamma_1 \in H^*_{CR, cs}(Y)$, $\gamma_2, \ldots, \gamma_n \in H^*_{CR}(Y)$, and   $\alpha_1, \ldots, \alpha_k \in H^*([V/G])$, define  $\br{ \gamma_1 \psi^{k_1}, \ldots, \gamma_n \psi^{k_n}| \alpha_1, \ldots, \alpha_k}_{h,n|k,d}^{Y}$ to be 
$$ \int_{[Q_{h,n}(Y, d)]^{vir}} 
\bbr \cdot
{ev_1^c}^*(\gamma_1) \cup \psi_1^{k_1} \cup \prod_{2 \leq i \leq n}  \left( ev_i^*(\gamma_i) \cup \psi_i^{k_i}\right) \cup \prod_{j=1}^k \hat{ev}^{Q *}(\alpha_j), $$
where ${ev_1^c}^*$ is the pullback in compactly supported cohomology.

Given $\gamma_1, \gamma_2 \in H^*_{CR, ct}(Y)$, $\gamma_3, \ldots, \gamma_n \in H^*_{CR}(Y)$, and   $\alpha_1, \ldots, \alpha_k \in H^*([V/G])$, define
$$\br{ \gamma_1 \psi^{k_1}, \ldots, \gamma_n \psi^{k_n}| \alpha_1, \ldots, \alpha_k}_{h,n|k,d}^{Y}$$ to be
$$ \int_{[Q_{h,n|k}(Y, d)]^{vir}}  
\bbr \cdot
{ev_1}^*(\gamma_1) \cup_{cs} {ev_2}^*(\gamma_2) \cup \psi_1^{k_1} \cup \psi_2^{k_2} \cup \prod_{3 \leq i \leq n}   \left(ev_i^*(\gamma_i) \cup \psi_i^{k_i} \right)\cup \prod_{j=1}^k \hat{ev}^{Q *}_j(\alpha_j). $$
\end{definition}

\subsection{GLSM invariants} \label{ss:GLSMlight}
In this section we use the same type of modification of the target $[V/G]$ to define LG quasimaps with light points and their corresponding GLSM invariants in the special case of $h=0, n=2$.

\begin{definition}\label{d:qlight} Define the space of genus $h$, $n$-pointed, degree $d$ LG quasimaps with $k$ light points to be
$$QLG_{h,n|k}(Y, d):= QLG_{h,n}(Y \times (\PP^0)^k, (d,1^k)).$$
\end{definition}
On the moduli space of LG quasimaps, there exist light evaluation maps to $[V_0 /(G/\langle J\rangle)]$, where $V_0$ is the $\CC^*_R$-fixed locus of $V$.   The evaluation map at the $j$th light point is defined as 
\begin{equation}\label{e:lightev}\hat {ev}^{QLG}_j\colon   QLG_{h,n|k}(Y, d)\xrightarrow{\hat s_j} \cC \xrightarrow{[u_\Gamma]} [V/\Gamma] \xrightarrow{\op{proj}} [V_0/\Gamma] \to[V_0 /(G/\langle J\rangle)],\end{equation}
where the last map is induced by the homomorphism $\Gamma \to G/\langle J \rangle$ (note that $\langle J \rangle$ acts trivially on $V_0$).

\begin{remark}[Motivating the space of LG quasimaps with light points]\label{altcomp} 
Let $Y^0_J$ denote the $\CC^*_R$-invariant locus of the $J$th twisted sector $Y_J$.  Assume that  
$Y^0_J = [V_0 \sslash_\theta G]$ 
and 
consider the substack of $QLG_{h,n+k}(Y, d)$ of LG quasimaps where the last $k$ marked points each map via the evaluation map to $Y^0_J$:
$$QLG_{h,n+k}(Y, d)(Y^0_J)_{j=1}^k := \bigcap_{j=1}^k ev_j^{-1}( Y^0_J).$$  The space $QLG_{h,n+k}(Y, d)(Y^0_J)_{j=1}^k$ contains an open subset $QLG_{h,n+k}(Y, d)^\circ(Y^0_J)_{j=1}^k$ defined by the condition that $\omega_{\cC}(p_1 + \cdots + p_n) \otimes L_\theta^{\otimes \epsilon}$  is ample for all $\epsilon >0 $.

On the other hand, let $QLG_{h,n|k}(Y, d)^\circ$ denote the open substack of $QLG_{h,n|k}(Y, d)$ consisting of quasimaps for which the light points are disjoint from each other and from basepoints, and $\omega_{\cC}(p_1 + \cdots + p_n) \otimes L_\theta^{\epsilon}$ is ample for all $\epsilon >0 $.

Following the arguments of
\cite[Section~5.4]{FKim}, by applying a \emph{Hecke modification} and forgetting the orbifold structure at the last $k$ marked points, one can define a map
$$QLG_{h,n+k}(Y, d)^\circ(Y^0_J)_{j=1}^k \to QLG_{h,n|k}(Y, d)^\circ$$
which, due to the gerbe structure at the last $k$ marked points, is a $\mu_{d_w}^k$-gerbe (in particular the map on coarse spaces is an isomorphism).

Furthermore, one can check using the proof of \cite[Lemma~5.5 (1)]{FKim} that the evaluation maps are compatible, i.e., for $1 \leq j\leq k$ the following diagram commutes,
\[
\begin{tikzcd}
QLG_{h,n+k}(Y, d)^\circ(Y^0_J)_{j=1}^k \ar[r]\ar[d, "ev_{n+j}"] & QLG_{h,n|k}(Y, d)^\circ \ar[d, "\hat{ev}^{QLG}_j"] \\
Y^0_J =\left[(V_0)\sslash_\theta G\right] \ar[r] &\left[(V_0) \sslash_\theta (G/\langle J \rangle) \right] \subset \left[V_0/ (G/\langle J \rangle) \right].
\end{tikzcd}
\]

Thus, under mild assumptions, the space $QLG_{h,n|k}(Y, d)$ may be viewed, up to a $\mu_{d_w}^k$-gerbe structure, as an alternative compactification of $QLG_{h,n+k}(Y, d)^\circ(Y^0_j)$ where the last $k$ points are allowed to coincide.  
This is the motivation behind   Definition~\ref{d:qlight}.  We expect these invariants to be related to GLSM invariants with all heavy marked points via a wall-crossing formula as in \cite{Pin, Zhou2}.
\end{remark}

Let $j\colon [V_0/(G/\langle J\rangle)] \to [ V^J/(G/\langle J\rangle)]$ denote the obvious inclusion of stacks.
The $(G/\langle J\rangle)$-equivariant analogue of \eqref{e:BMduality} and Definition~\ref{d:BMpush} defines pushforwards
$$j_*\colon H^*([V_0/(G/\langle J\rangle)]) \to H^*([ V^J/(G/\langle J\rangle)])$$
and 
$$j'_*\colon H^*( [V_0/(G/\langle J\rangle)]) \to H^*([ V^J/(G/\langle J\rangle)], w^{+\infty}).$$
Viewing $j$ as the zero section of a vector bundle, we note that $j_*\circ j^*(\gamma)= e(V_{d_w}) \cup \gamma$ where here $V_{d_w}$ denotes the vector bundle $[(V^J \times V_{d_w})/(G/\langle J\rangle)] \to [V^J/(G/\langle J\rangle)]$.
As before, if $$\phi^w\colon H^*([ V^J/(G/\langle J\rangle)], w^{+\infty}) \to H^*([ V^J/(G/\langle J\rangle)])$$ is the $(G/\langle J\rangle)$-equivariant pullback of the map of pairs $(V_0, \emptyset) \hookrightarrow (V^J, w^{+\infty})$, then 
\begin{equation} \label{e:jjp}
j_* = \phi^w \circ j'_*.\end{equation}

%{\color{blue}https://people.math.osu.edu/anderson.2804/eilenberg/lecture2.pdf for approximation spaces, for relative groups, use prop 1.1 of loc. cit. and the relative long exact sequence.  
%%PS: Defining equivariant Borel-Moore homology requires a little more care, since the spaces ????×?????????
% are infinite-dimensional fiber bundles. But they have finite-dimensional approximations ??????×???????????, so it makes sense to define
%????????=????+dim??????(??????×????)
%for ???0
%m?0. The equivariant Poincaré isomorphism is just the ordinary one for these approximation spaces.
%}

\begin{definition} \label{d:lightinvts}
 Given classes $\gamma_1, \gamma_2 \in \cc H_{ct}(Y, w)$ and 
 $$\tilde \sigma_1, \ldots, \tilde \sigma_k \in H^*([V/(G/\langle J\rangle)], w^{+\infty})$$ arising as the pushforward of $\sigma_1, \ldots, \sigma_k \in H^*([V_0/(G/\langle J\rangle)])$ via $j'_*$,
 define $\br{\gamma_1 \psi^{k_1}, \gamma_2\psi^{k_2}| \tilde \sigma_1, \ldots, \tilde \sigma_k}_{0,2|k,d}^{(Y, w)}$ to be 
\[ \int_{ [QLG_{0,2|k}(Y, d)]^{vir}} \bbr \cdot ev_1^*(\phi^{cs}(\gamma_1)) \cup_{cs} ev_2^*(\phi^{cs}(\gamma_2)) \cup \psi^{k_1} \cup \psi^{k_2} \cup \prod_{j=1}^k \hat{ev}^{QLG *}_j(\sigma_j). \]
\end{definition}
In the toric setting under a natural assumption, we verify that the above invariants are well-defined.  This is the main setting in which we will do computations (see Corollary~\ref{c:Ifctn}).
\begin{proposition}\label{t:lightwd}
Suppose that $G = (\CC^*)^k$ and $(\Sym^\bullet(V^\vee))^\Gamma = \CC$.  Then
the above definition does not depend on the particular $\sigma_i$ chosen such that $j'_*(\sigma_i) = \tilde \sigma_i$.  
\end{proposition}

\begin{proof}
We claim that  $j'_*$ is injective. By \eqref{e:jjp}, $j'_*$ is injective if $j_*$ is injective.  Note that $$j^*\colon H^*([V^J/(G/\langle J\rangle)]) \to H^*([V_0/(G/\langle J\rangle)])$$ is an isomorphism, thus the claim holds whenever $j_* \circ j^* = e(V_{d_w}) \cup ( - )$ is injective. 

Because $G$ is a torus, $V_{d_w}$ splits as a sum of one-dimensional representations.  We claim that  none of these are trivial.  Suppose on the contrary that $V_{d_w}$ contains a trivial rank-one $G$-representation $L'$ as a summand.  Let $p$ be the degree $d_w$ homogeneous coordinate on $L'$.  Because $w$ is degree $d_w$,  we may write it as
$$w = p \cdot w_p + w'$$
where $w_p$ is a an element of $\Sym^\bullet(V_0^\vee)$ and $w'$ does not depend on $p$.  Because $p$ is $G$-invariant and each monomial of $W$ is $G$-invariant, we conclude that $w_p$ is $G$-invariant.   The assumption $(\Sym^\bullet(V^\vee))^\Gamma = \CC$ implies that $\Sym^\bullet(V_0^\vee)^G = \CC$. %, and consequently, that $[V_0//G]$ is proper.
Consequently $w_p$ is constant.  If $w_p \neq 0$ then the critical locus $Z(dw)$ is empty as it lies in $\partial w/\partial p = w_p =  0$.  Thus $w_p$ must be zero, so $w = w'$.  Because $L'$ is a trivial representation, the critical locus $Z(dw)$ is therefore isomorphic to $Z(dw'|_{p=0}) \times \CC$
where $Z(dw'|_{p=0})$ is the critical locus in $\{p = 0\} \subset V$.  This contradicts the assumption that $Z(dw)$ is proper over $\spec \CC$.

Because $V_{d_w}$ has no trivial summands, its Euler class is a nonzero polynomial in $H^*([ V^J/(G/\langle J\rangle)]) \cong \CC[t_1, \ldots, t_k]$.  This proves that $j_* \circ j^* = e(V_{d_w}) \cup ( - )$ is injective. 
\end{proof}
\begin{remark}
The argument above holds for other groups $G$ as well, under the conditions on $G$ that $H^*(G/\langle J \rangle)$ embeds in an integral domain and the only representations with trivial Euler class contain a trivial summand.
\end{remark}

Let $i\colon [V_0/G] \to [V/G]$  denote the inclusion and $r\colon [V_0/G] \to [V_0/(G/\langle J\rangle)]$ denote the $\langle J \rangle$-gerbe.  The pullback $r^*$ is an isomorphism by the Lyndon-Hochschild-Serre spectral sequence for the extension 
$$1 \to \langle J\rangle \to G \to G/\langle J\rangle \to 1$$
and the fact that $H^*(B\langle J\rangle, \CC) = \CC$.  Define the map 
\begin{align*}
\tau\colon H^*([V/G]) \to& H^*([ V^J/(G/\langle J\rangle)], w^{+\infty}) \\
 \alpha  \mapsto &j'_*((r^*)^{-1}i^*\alpha).
\end{align*}
We have the following comparison between quasimap invariants and GLSM invariants with light points:
\begin{theorem}\label{t:complight}
 Given classes $\gamma_1, \gamma_2 \in \cc H_{ct}(Y, w)$ and $\alpha_1, \ldots, \alpha_k \in H^*([V/G])$, 
\[\br{\gamma_1 \psi^{k_1}, \gamma_2\psi^{k_2}| \tau(\alpha_1), \ldots, \tau(\alpha_k)}_{0,2|k,d}^{(Y, w)} =  \br{\phi^w(\gamma_1)\psi_1^{k_1},\phi^w(\gamma_2)\psi_2^{k_2}| \alpha_1, \ldots, \alpha_k}_{0,2|k,d}^{Y}.\]
\end{theorem}
\begin{proof}
First recall that by definition, 
\begin{align*}
QLG_{0,2|k}(Y, d) &= QLG_{0,2}(Y \times (\PP^0)^k, (d,1^k))\\
Q_{0,2|k}(Y, d) & = Q_{0,2}(Y \times (\PP^0)^k, (d,1^k)).
\end{align*}
By Proposition~\ref{p:modint} applied to the spaces on the right, $QLG_{0,2|k}(Y, d)$ and $Q_{0,2|k}(Y, d)$ are canonically isomorphic and their virtual classes are identified.  

Next we compare the light evaluation maps.  For $1 \leq j\leq k$, we have the following commuting diagram:
\[
\begin{tikzcd}
 Q_{0,2|k}(Y, d)\ar[d, "\hat{ev}^Q_j = {\left[u_G\right]} \circ \hat s_j"] \ar[r, "\cong"] &QLG_{0,2|k}(Y, d) \ar[d, swap, "{\left[u_\Gamma \right]} \circ \hat s_j"]  \ar[ddd, bend left, "\hat{ev}^{QLG}_j"]\\
 \left[V/G \right] \ar[d, swap, "\pi"] \ar[r] & \left[V/\Gamma \right] \ar[d, "\pi"] \\
\ar[u, bend right, swap, "i"]  \left[V_0/G \right] \ar[d, "r"] \ar[r] & \left[V_0/\Gamma \right] \ar[d, "\pi"] 
  \\
\left[V_0/(G/\langle J\rangle)\right]  \ar[r, "="]& \left[V_0/(\Gamma/\CC^*_R)\right]
\end{tikzcd}
\]
We see that  $$ \hat{ev}_j^{Q *} \alpha = \hat{ev}_j^{QLG *}(r^*)^{-1}i^*\alpha.$$
The result then follows from 
Definition~\ref{d:lightinvts}.

\end{proof}

%\begin{remark}
%To ease notation, for the remainder of the paper we will denote both $ \hat{ev}_j^{Q}\colon Q_{h,n|k}(Y, d) \to[V /G]$ and $\hat{ev}_j^{QLG} \colon QLG_{h,n|k}(Y, d) \to[V_0 /(G/\langle J\rangle)]$  by $\hat{ev}_j$.  Which of the two maps is meant will be clear from context.
%\end{remark}

\section{Generating functions}\label{s:GF}
In this section we show how certain derivatives of $I$-functions for $Y$ can be used to obtain generating functions of GLSM invariants for $(Y, w)$.  In the toric setting, we provide an explicit formula for GLSM $I$-functions.
\subsection{The $S$-operator}
\begin{definition}
Let $\CC[[q]]$ denote the \emph{Novikov ring}:
$$\CC[[q]] := \CC[[ \eff(V, G, \theta)]],$$ where we denote the element $d \in \CC[[ \eff(V, G, \theta)]]$ by $q^d$.

Let $\bt = \sum_{1 \leq j \leq m} t^j T_j$ where $T_j\in H^*([V/G])$ and $t^j$ are formal parameters.  For $\gamma_1 \in H^*_{CR, cs}(Y)$, $ \gamma_2 \in H^*_{CR}(Y)$ and $k_1, k_2 \geq 0$, define the generating function
\[\brr{ \gamma_1 \psi^{k_1}, \gamma_2\psi^{k_2}}_{0,2}^{Y}(\bt) := \sum_{d \in\eff(V, G, \theta)} \sum_{k \geq 0} \frac{q^d}{k!} \br{ \gamma_1\psi^{k_1}, \gamma_2\psi^{k_2} | \bt, \ldots, \bt}_{0, 2|k, d}^Y,\]
where 
\[\br{ \gamma_1\psi^{k_1}, \gamma_2\psi^{k_2} | \bt, \ldots, \bt}_{0, 2|k, d}^Y := \int_{\left[Q_{0, 2|k}(Y, d)\right]^{vir}} \bbr \cdot
{ev_1^c}^*(\gamma_1) \cup ev_2^*(\gamma_2)
\cup \psi_1^{k_1} \cup \psi_2^{k_2}\prod_{j=1}^{k} \hat{ev}_j^{Q *}(\bt).\]
Alternatively, for $\gamma_1, \gamma_2 \in H^*_{CR,ct}(Y)$, $\brr{ \gamma_1\psi^{k_1}, \gamma_2\psi^{k_2}}_{0,2}^{Y}(\bt)$ is defined similarly, but replacing ${ev_1^c}^*(\gamma_1) \cup ev_2^*(\gamma_2)$ with 
$ev_1^*(\gamma_1) \cup_{cs} ev_2^*(\gamma_2)$.
\end{definition}

\begin{definition}\label{d:SY}
Let $\{ \tilde \gamma_i\}_{i \in \tilde I}$ be a basis for $H^*_{CR}(Y)$ and let 
$\{ \tilde \gamma^i\}_{i \in \tilde I}$ be the dual basis  in $H^*_{CR,cs}(Y)$.
The (light point, quasimap) $S$-operator  on $H^*_{CR}(Y)[[q]][[t^i]][z, z^{-1}]$
is defined by
\[S^{Y}(q, \bt, z)(\gamma) := \sum_{i \in \tilde I} \tilde \gamma_i \brr{  \frac{\tilde \gamma^i}{z- \psi}, \gamma}_{0, 2}^{Y}(\bt).\] 
As shown in \cite{CKbig}, the inverse of $S$ is given by:
$L^{Y}(q, \bt, z) = {S^Y}^*(q, \bt, -z)$,
\[L^{Y}(q, \bt, z)(\gamma) := \sum_{i \in \tilde I} \tilde \gamma_i \brr{\tilde \gamma^i, \frac{\gamma}{-z- \psi}}_{0, 2}^{Y}(\bt).\]
\end{definition}
Recall from Section~\ref{s:ctcpp} that there is a perfect pairing on $H^*_{CR,ct}(Y)$.  Let $\{ \gamma_i\}_{i \in I}$ be a basis for $H^*_{CR,ct}(Y)$ and let 
$\{ \gamma^i\}_{i \in I}$ be the dual basis. 
\begin{lemma}\label{l:ct}
The operator $S^{Y}(q, \bt, z)$ preserves the compact type subspace. I.e.,
If $\gamma \in H^*_{CR,ct}(Y)[[q]][[t^i]][z, z^{-1}]$, then $S^{Y}(q, \bt, z)(\gamma)\in H^*_{CR,ct}(Y)[[q]][[t^i]][z, z^{-1}]$.  The same holds for $L^{Y}(q, \bt, z)$.
For   $\gamma$ as above, we have  
$$S^{Y}(q, \bt, z)(\gamma) = \sum_{i \in  I}  \gamma_i \brr{ \frac{\gamma^i}{z- \psi}, \gamma}_{0, 2}^{Y}(\bt).$$
\end{lemma}
\begin{proof}
The generating function $S^{Y}(q, \bt, z)(\gamma)$ may be alternatively defined as 
\[\gamma + \sum_{d \in \op{Eff}} \sum_{k \geq 0} \frac{q^d}{k!}
 \wt{ev_1}_* \left(\frac{ev_2^*(\gamma)}{z - \psi} \prod_{j=1}^{k} \hat{ev}_j^{Q *}(\bt) \cap \left[Q_{0, 2|k}(Y, d)\right]^{vir} \right),\]
where $\wt{ev_1}$ is the composition of ${ev_1}$ with the involution $I\colon IY \to IY$.
Both $ev_1$ and $ev_2$ are proper by Lemma~\ref{l:propev}.  By \cite[Proposition~2.5]{Shoe1} $\wt{ev_1}_*$ and ${ev_2}^*$ preserve cohomology of compact type.  This proves the first statement.
The same argument applies to $L^{Y}(q, \bt, z)$.  

The final statement follows from the first together with the 
observation that, for $\tilde \gamma^i \in H^*_{CR, cs}(Y)$ and $\gamma \in H^*_{CR, ct}(Y)$, 
$${ev_1^c}^*(\tilde \gamma^i) \cup ev_2^*(\gamma) = {ev_1}^*(\phi^{cs}(\tilde \gamma^i)) \cup_{cs} ev_2^*(\gamma).$$   
\end{proof}

For the remainder of the paper, we will assume the following:
\begin{assumption}\label{a:littleas} 
 The compact type subspace $H^*_{CR,ct}(Y)$ is spanned by the pushforwards from smooth proper substacks $Z$ lying in $w^{-1}(0)$.  \end{assumption}
If Assumption~\ref{a:littleas} holds, then $ \phi^w \left(\cc H_{ct}(Y, w)\right) = H^*_{CR,ct}(Y).$  We expect it to hold quite generally.  

\begin{lemma}
The assumption holds for $Y$ a toric stack ($G \cong (\CC^*)^k$).
\end{lemma}
\begin{proof}
We first consider the case that $Y = V \sslash_\theta G$ is a simplicial toric \emph{variety}. Let $V = \spec \left( \CC[x_i]_{\{1\leq i\leq r\}}\right)$.  
Given $I \subset \{1, \ldots, r\}$, let $Y_I \subset Y$ denote the closed subspace defined by the equations $\{x_i = 0\}_{i \in I}$.

Let $\Sigma$ denote the toric fan for $Y$, which may be obtained from the GIT description by, e.g., \cite{CIJ}.
Following the description of Borisov--Horja of compactly supported cohomology \cite[Proposition~2.4]{BHHMS}, the compact type cohomology $H^*_{ct}(Y)$ is generated by the images of the maps
$H^*(Y_I) \to H^*(Y)$
for $I $ ranging over 
all subsets of $ \{1, \ldots, r\}$ such that $\sigma_I := \op{cone}\{v_i\}_{i \in I}$ is a cone of $\Sigma$ and the interior $\sigma_I^\circ$ is contained in $|\Sigma|^\circ$.  One can check,  for these choices of $I$, that $Y_I$ is proper over $\spec \CC$.  
%THIS IS BECAUSE $\Sigma/\sigma_I$ IS SEEN TO BE COMPLETE.  FROM THE DESCRIPTION OF I, ITS IMMEDIATE THAT $\RR^{R-|I|}$ IS SPANNED BY POSITIVE COMBINATIONS OF RAYS IN $\Sigma/\sigma_I$
So to verify Assumption~\ref{a:littleas} in this case, it suffices to show that for each such choice of $I$, the restriction of $w$ to $Y_I$ 
is identically zero.  

$Y_I$ may be constructed as a GIT quotient of $\spec \left( \CC[x_i]_{\{i \notin I\}}\right)$ by $G$.  Because $Y_I$ is proper over $\spec \CC$, any algebraic function on $Y_I$ is constant.  As a result, any $G$-invariant function in $\CC[x_i]_{\{i \notin I\}}$ is constant on a nonempty open subset of $\spec \left( \CC[x_i]_{\{i \notin I\}}\right)$, and therefore on all of $\spec \left( \CC[x_i]_{\{i \notin I\}}\right)$.   It follows that that $\CC[x_i]_{\{i \notin I\}}^G = \CC$.
Because $w$ is homogeneous of degree $d_w > 0$ with respect to the $R$-charge, it contains no constant term.  Therefore $w$ is zero on $\spec \CC$.  Since $w|_{Y_I}: Y_I \to \CC$ factors through $\spec \left(\CC = \CC[x_i]_{\{i \notin I\}}^G\right)$, we conclude that $Y_I \subset w^{-1}(0)$.

In the general case when $Y$ is a stack, we must consider the Chen--Ruan cohomology $H^*_{CR,ct}(Y) = H^*_{ct}(IY)$.  
Observe that:
\begin{enumerate}
\item each component $Y_g \subset IY$ is itself a toric stack arising as a GIT quotient;
\item the cohomology of $Y_g$ (with coefficients in $\CC$) is equal to the cohomology of its coarse space, the simplicial toric variety $|Y_g|$.
\end{enumerate}
Working component by component, the result then follows from the previous case.
\end{proof}

By Assumption~\ref{a:littleas}, the map $\phi^w\colon \cc H_{ct}(Y, w) \to H^*_{CR,ct}(Y)$ is surjective.  Choose a splitting 
$$\sigma^w\colon H^*_{CR,ct}(Y) \to \cc H_{ct}(Y, w).$$
Denote the image of $\sigma^w$ by $ \cc R_c(Y, w)$.
Although there is no canonical choice of $ \cc R_c(Y, w)$, by the broad vanishing result of Corollary~\ref{c:bv}, the $h=0$, $n=2$ GLSM invariants with heavy point insertions in $\cc H_{ct}(Y, w)$ are fully determined by the GLSM invariants with insertions from $ \cc R_c(Y, w)$.

\begin{definition}\label{d:SYw}
Let $\{ \gamma_i\}_{i \in I}$ be a basis for $H^*_{CR,ct}(Y)$ and let 
$\{ \gamma^i\}_{i \in I}$ be the dual basis. 
Define the (light point, LG quasimap)  $S$-operator on $\cc H_{ct}(Y, w)[[q]][[t^i]][z, z^{-1}]$ by:
\[S^{(Y, w)}(q, \bt, z)(\gamma) := \sum_{i \in I} \sigma^w(\gamma_i) \brr{ \frac{ \sigma^w(\gamma^i)}{z- \psi}, \gamma}_{0, 2}^{(Y, w)}(\bt),\]
where $\bt = \sum_{1 \leq j \leq m} t^j T_j$ with $T_j\in H^*([ V^J/(G/\langle J\rangle)], w^{+\infty})$,  and the variables $t^j$ are formal parameters.   
\end{definition}
\begin{proposition}\label{p:Srel} Whenever $P(q, \bt, z) \in H^*_{CR, ct}(Y)[[q]][[t^i]][z, z^{-1}]$,
\[S^{(Y, w)}(q, \tau(\bt), z)(\sigma^w(P(q, \bt, z))) = \sigma^w(S^{Y}(q, \bt, z)(P(q, \bt, z))).\]
\end{proposition}

\begin{proof}
This follows immediately from Theorem~\ref{t:complight}.
\end{proof}

\subsection{$I$-functions}

\begin{definition}\label{d:BigI}
A $(0^+, 0^+)$ \emph{$I$-function} for $Y$ is any function $I(q, \bt, z)$ of the form 
$$S^{Y}(q, f(\bt), z)P(q, \bt, z)$$
where $f \in H^*([V/G])[[t^i]]$ is a formal  cohomology-valued function satisfying $f(0) = 0$, and  $P(q, \bt, z)$ lies in $H^*_{CR}(Y)[[q]][[t^i]][z]$ (in particular it contains only \emph{positive} powers of $z$).

\end{definition}
\begin{remark}
The notation ``$(0^+, 0^+)$'' comes from \cite{CKbig}, where $S^Y$ is denoted by $\mathbb{S}^{(0^+ \theta, 0^+)}$.
The first $0^+$ in the expression indicates that we are considering $0^+$-stable quasimap invariants of $[V\sslash_\theta G]$ and the second indicates that all marked points after the first two are light points.
\end{remark}

\begin{remark}
Let $S^Y_{\op{GW}}$ denote the analogous operator to $S^Y$, but replacing quasimap invariants with Gromov--Witten invariants and replacing light points with heavy points \cite[(10.14)]{CK}.  Then $S^Y_{\op{GW}}(q, \bt, z)(\mathbf 1)$ is Givental's $J$-function, $J^Y(q, \bt, z)$.   In the language of Givental's symplectic formalism, if a function $I(\bt ,z)$ takes the form $S^Y_{\op{GW}}(q, f(\bt), z)P(q, \bt, z)$ with
$P(q, \bt, z)$  in $H^*_{CR}(Y)[[q]][[t^i]][z]$, then $I(\bt ,z)$ lies on a tangent space $T_\tau \cc L^Y$  of the overruled Lagrangian cone $\cc L^Y$ for $\tau = J^Y(q, f(\bt), z)$.  
Results on $\epsilon$-wall crossing such as \cite[Theorem~7.3.1]{CK1} and \cite[Theorem~4.6]{CCKbig} suggest that $(0^+, 0^+)$ $I$-functions will also lie on a tangent space $T_\tau \cc L^Y$ for some $\tau$.
This is the motivation behind Definition~\ref{d:BigI}.

The relationship between 
the $I$-functions for toric stacks defined in \cite{CCKbig, CCIT2}
and other $(0^+, 0^+)$ $I$-functions
will be clarified in Proposition~\ref{p:Iderv}.  The connection to GLSM $I$-functions (as defined below) will be given in Theorem~\ref{t:dI}.
\end{remark}

Definition~\ref{d:BigI}  can be generalized to the setting of GLSMs.  As usual we restrict our attention to compact-type insertions.
\begin{definition}\label{d:BigIGLSM}
A $(0^+, 0^+)$ GLSM \emph{$I$-function} for $(V, G, \theta, w)$ is any function of the form %\note{cY vs Y}
$$S^{(Y, w)}(q, f(\bt), z)P(q, \bt, z)$$
where $f \in H^*([ V^J/(G/\langle J\rangle)], w^{+\infty})[[t^i]]$ is a formal  cohomology-valued function satisfying $f(0) = 0$, and $P(q, \bt, z) $ lies in $ \cc H_{ct}(Y, w)[[q]][[t^i]][z].$
\end{definition}

\begin{lemma}\label{l:Ict}
If a $(0^+, 0^+)$ $I$-function for $Y$ $I(q, \bt ,z)= S^{Y}(q, \bt, z) P(q, \bt, z)$ is supported in cohomology of compact type, then so is $P(q, \bt, z)$.\end{lemma}

\begin{proof} By \cite[Proposition~3.7]{CCKbig}, we have $P(q, \bt, z) = L^Yf(q, \bt, z)I(q, \bt ,z)$.   By Lemma~\ref{l:ct},  $L^{Y}(q, \bt, z)$ preserves the compact type subspace.
\end{proof}

Next we recall the definition of the \emph{Big $I$-function}, $\II^Y(q, \bt, z)$, of \cite{CKbig, CCKbig}.  The $0^+$-quasimap graph space $QG_{0,n|k}(Y, d)$ is the space of genus zero quasimaps (with $k$ light points) to $[V\sslash_\theta G] \times [\CC^2 \sslash_{\op{id}} \CC^*]$ of degree $(d, 1)$which are $0^+$-stable on the first factor and $\infty$-stable (Kontsevich-stable) on the second factor:
$$QG_{0,n|k}(Y, d) := Q_{0,n|k}^{0^+, \infty}(Y \times \PP^1, (d,1)).$$
The graph space may be viewed as the moduli space of quasimaps to $Y$ such that each source curve has a distinguished component whose coarse underlying curve is equipped with a parametrization, induced by the degree one map to $\PP^1$.  

Consider the $\CC^*$-action on $\PP^1$ which scales the first homogeneous coordinate.  This induces a natural 
 $\CC^*$-action on the graph space.  The evaluation maps
 $ev_i:QG_{0,n|k}(Y, d) \to IY$ for $1 \leq i \leq n$ and $\hat{ev}_j^{Q}\colon QG_{0,n|k}(Y, d) \to [V/G]$ for $1 \leq j \leq k$ are 
 equivariant, where the action of $\CC^*$ on $IY$ and $[V/G]$ is trivial.  
  Let $z \in H^*_{\CC^*}(pt)$
   denote the equivariant parameter $$z := e_{\CC^*}(T_0\PP^1),$$ the first Chern class of $T\PP^1$ at $0 \in \PP^1$.   Denote by $p_0$ and $p_\infty$ the classes in $H^*_{\CC^*}(\PP^1)$ determined by the conditions:
$$p_0|_0 = z, p_0|_\infty = 0, p_\infty|_0 = 0, p_\infty|_\infty = -z.$$

Fix $ n=1$.  Define $F^{k,d}_{1,0}$ to be the $\CC^*$-fixed locus such that the source curve is irreducible, the heavy marked point $p_1$ lies over $\infty$, there is a basepoint of degree $d(L_\theta)$ over $0$, and all light points are concentrated over $0$.  Let $\iota_d\colon F^{k,d}_{1,0} \to QG_{0,n|k}(Y, d)$ denote the inclusion.  
\begin{definition}
Define 
$$ \op{Res}_{F^{k,d}_{1,0}}(\bt^k) := \frac{ev_1^*(p_\infty) \cup \prod_{j=1}^k \hat{ev}_j^{Q *}(\bt) \cap [F^{k,d}_{1,0}]^{vir}}{e_{\CC^*}\left( N^{vir}_{F^{k,d}_{1,0}}\right)},
$$
where $N^{vir}_{F^{k,d}_{1,0}}$ is the virtual normal bundle in the sense of \cite{GrP}.
Define
\begin{equation}\label{e:I}
\II^Y(q, \bt, z) := \sum_{d \in\eff(V, G, \theta)} \sum_{k \geq 0} \frac{q^d}{k!} (\wt{ev_1} \circ {\iota_d})_*\op{Res}_{F^{k,d}_{1,0}}(\bt^k).
\end{equation} 
We view $\II^Y(q, \bt, z)$ as a power series in the $q$ and $t^i$ variables, whose coefficients lie in the localized equivariant cohomology ring
$H^*_{CR, \CC^*, \op{loc}}(Y) = H^*_{CR}(Y)[z, z^{-1}].$
\end{definition}
By \cite[Remark~4.3]{CCKbig}, $\II^Y(q, \bt, z)$ is equal to the $I$-function defined in \cite[Definition~4.1]{CCKbig}.  In the toric case, if $\bt \in H^2([V/G])$, this coincides with the ``S-extended stacky $I$-function'' of \cite{CCIT2}.
By  \cite[(4.2)]{CCKbig},  it is a $(0^+, 0^+)$ $I$-function in the sense of Definition~\ref{d:BigI}.

Given $\rho \in H^*([V/G])$, define 
$$z\partial_\rho \II^Y(q, \bt, z) := z\frac{\partial}{\partial s}  \II^Y(q, \bt + s \rho, z)|_{s = 0}.$$ 
\begin{proposition}\label{p:Iderv}
The derivative $z\partial_\rho \II^Y$ is a $(0^+, 0^+)$ $I$-function for $Y$.
\end{proposition}
\begin{proof}
This was stated in (2.8) of \cite{KiLh} in the case where the target is projective space.  The proof follows  from localization on the graph space as in Section~5.4 of \cite{CK1}.  We outline the argument here.

Let $\{ \tilde \gamma_i\}_{i \in \tilde I}$ be a basis for $H^*_{CR}(Y)$ and let 
$\{ \tilde \gamma^i\}_{i \in \tilde I}$ be the dual basis  in $H^*_{CR,cs}(Y)$.  
Define $$P_{\rho}(q, \bt, z) := \sum_{i \in \tilde I} \tilde \gamma_i \brr{\tilde \gamma^i \otimes p_\infty| \rho \otimes p_0}_{0, 1|1}^{QG(Y)^{0^+, 0^+}}(\bt),$$ 
where $\brr{ -}^{QG(Y)^{0^+, 0^+}}(\bt)$ is the double bracket generating function of $(0^+, 0^+)-stable$ \emph{graph space} invariants as in \cite[Section~2.2]{KL}.
 Because $P_{\rho}(q, \bt, z)$ is defined before localization, it lies in $H^*_{CR}(Y)[[q]][[t^i]][z]$ (i.e. it contains only positive powers of the equivariant parameter $z$).

 After localization on the graph space, we obtain
$$P_\rho(q, \bt, z) = {S^{Y}}^*(q, \bt, -z)(z \partial_\rho \II^Y(q, \bt, z)).$$  Applying the operator $S^{Y}(q, \bt, z)$ to both sides, we conclude: 
$$z \partial_\rho \II^Y(q, \bt, z) = S^{Y}(q, \bt, z) P_\rho(q, \bt, z).$$
\end{proof}

We arrive at one of the main theorems of the paper, which describes how $(0^+, 0^+)$ GLSM $I$-functions arise as derivatives of the big $I$-function of $Y$. 
\begin{theorem}\label{t:dI}  
 Given a class $\rho \in H^*([V/G])$, if
$z \partial_\rho \II^Y(q, \bt, z)$ lies in cohomology of compact type, then $\sigma^w\left(z \partial_\rho \II^Y(q, \bt, z)\right)$ is a $(0^+, 0^+)$ GLSM $I$-function for $(V,G, \theta, w)$.
\end{theorem}

\begin{proof}
By Proposition~\ref{p:Iderv}, we can write $z \partial_\rho \II^Y(q, \bt, z)$ as $S^{Y}(q, \bt, z) P_\rho(q, \bt, z)$ for  $ P_\rho(q, \bt, z) \in H^*_{CR}(Y)[[q]][[t^i]][z]$.  
By Lemma~\ref{l:Ict}, $ P_\rho(q, \bt, z)$ lies in compact type cohomology: $ P_\rho(q, \bt, z) \in H^*_{CR, ct}(Y)[[q]][[t^i]][z]$.  Applying Proposition~\ref{p:Srel}, we conclude that 
\begin{align*}
\sigma^w\left(z \partial_\rho \II^Y(q, \bt, z)\right) &=  \sigma^w\left(S^{Y}(q, \bt, z) P_\rho(q, \bt, z)\right) \\
&= S^{(Y, w)}(q, \tau(\bt), z)(\sigma^w( P_\rho(q, \bt, z))),
\end{align*}
which is of the form given in Definition~\ref{d:BigIGLSM}.
\end{proof}

\subsection{The toric case}\label{s:toriccase}
In this section we focus on the case that $G$ is a torus, and give specific conditions for when $z \partial_\rho \II^Y(q, \bt, z)$ lies in cohomology of compact type.
We then use Theorem~\ref{t:dI} and the explicit formula for $\II^Y(q, \bt, z)$ from \cite{CCKbig} to obtain an explicit $I$-function for any toric GLSM.

Let $(V, G, \theta, w)$ be a GLSM such that
 $G = (\CC^*)^k$.  Let $$\rho_1, \ldots, \rho_r \in \widehat G $$ denote the characters of $G$ which define the action on  $V = \CC^r.$ By abuse of notation we will also use  $\rho_1, \ldots, \rho_r $ to denote the corresponding cohomology  classes in $H^2([V/G])$.
 We will identify the coordinates $x_1, \ldots, x_r$ on $V$ with the homogeneous coordinates on $[V/G]$ and $[V\sslash_\theta G]$.  

In \cite{CKbig, CCKbig}, $\II^Y$ is computed explicitly for toric varieties.  Choose characters $\eta_1, \ldots, \eta_l \in \widehat G$ which generate $H^*([V/G])$ as an algebra.  Then in the expression $\bt = \sum_{1 \leq j \leq m} t^j T_j$, each $T_j\in H^*([V/G])$ may be expressed as
$$T_j = p_j(\eta_1, \ldots, \eta_l) =: p_j(\eta_s),$$
for some polynomial $p_j(x_s):=p_j(x_1, \ldots, x_l)$ in $l$ variables.  For $g \in G$, let $\ii_g$ denote the fundamental class of the twisted sector $Y_g$, and for $d \in\eff(V, G, \theta)$, let
$$g_d = \left( e^{2 \pi i \langle d, \rho_1\rangle}, \ldots, e^{2 \pi i \langle d, \rho_r\rangle}\right) \in G.$$
With this notation, the big $I$-function is given by
\begin{align}\label{e:IY}
\II^Y(\bt, q, z) =& \sum_{d \in\eff(V, G, \theta)}q^d \exp \left(  \frac{1}{z} \sum_{j=1}^l t^j p_j\left(\eta_s + z\langle d, \eta_s\rangle\right) \right) \\
& \frac{\prod_{i|  \langle d, \rho_i\rangle< 0} \prod_{ \langle d, \rho_i\rangle \leq \nu< 0} (\rho_i + ( \langle d, \rho_i\rangle - \nu)z)}
{\prod_{i|  \langle d, \rho_i\rangle> 0} \prod_{ 0 \leq \nu < \langle d, \rho_i\rangle} (\rho_i + ( \langle d, \rho_i\rangle - \nu)z)} \ii_{g_d^{-1}}, \nonumber
\end{align}
where the products on the second line run over all \emph{integers} $\nu$ in the specified range \cite[Corollary~5.6 (2)]{CCKbig}.

\begin{lemma}\label{l:properlemma}
 Fix $d \in\eff(V, G, \theta)$ and $k \geq 0$.  The  composition $$F^{k,d}_{1,0} \xrightarrow{\iota_d} QG_{0, 1|k} (Y, d) \to \spec \left( \CC[x_i]_{\{1\leq i\leq r\}}^G\right)$$ factors through $\spec \left( \CC[x_i]_{\{i|\langle d, \rho_i\rangle = 0\}}^G\right)$.  
 \end{lemma}
\begin{proof}
By Theorem~2.7 of \cite{CCKbig}, $QG_{0, 1|k} (Y, d)$ is proper over $\spec \left( \CC[x_i]_{\{1\leq i\leq r\}}^G\right)$.  
Because $G$ is abelian, the  vector bundle $\cP \times_G V \to \cC$ splits as
$$\cP \times_G V = \prod_{i=1}^r \cc L_i.$$
Here $\cc L_i$ is the line bundle obtained as 
the pullback of  $[(V \times \CC_{\rho_i})/G] \to [V/G]$ via the map $\cC \to [V/G]$.  

Restrict the universal curve to $F^{k,d}_{1,0}$, denote it by $$\pi_F: \cC_F \to F^{k,d}_{1,0}.$$  By definition of $F^{k,d}_{1,0}$ and  $0^+$-stability, each fiber of $\pi_F$ is irreducible.  
If $\langle d, \rho_i\rangle < 0$, then $\cc L_i \to \cC_F$ has negative degree on each fiber.  Therefore 
the universal section of $\cc L_i \to \cC_F$ is zero, i.e. $\cC_F$ maps to the locus $\{x_i = 0\} \subset [V/G]$.  
Consequently, $F^{k,d}_{1,0}$ maps properly to the closed subset \\ $\spec \left( \CC[x_i]_{\{i|\langle d, \rho_i\rangle \geq 0\}}^G\right)$ of $\spec \left( \CC[x_i]_{\{1\leq i\leq r\}}^G\right)$.  

Next, we describe how the degree $d\in  \hom_\ZZ \left( \widehat G, \QQ\right)$ defines a subgroup of $G$.  By possibly replacing $d$ with a multiple $ n d$ for some $n \in \NN$, 
we can assume that $d $ lies in $\hom_\ZZ \left( \widehat G, \ZZ\right)$.  Consider the isomorphism
\begin{align*}
\hom\left(\CC^*, G\right)  &\to \hom_\ZZ \left( \widehat G, \ZZ\right) \\
f &\mapsto (\chi \mapsto \chi \circ f \in \hom(\CC^*, \CC^*) \cong \ZZ),
\end{align*}
which is canonical up to a choice of isomorphism $\hom(\CC^*, \CC^*) \cong \ZZ$.  Under this identification, $d(\CC^*)$ is a subgroup of $G$.  Consequently, 
$$\CC[x_i]_{\{i|\langle d, \rho_i\rangle \geq 0\}}^G \subset \CC[x_i]_{\{i|\langle d, \rho_i\rangle \geq 0\}}^{d(\CC^*)} = \CC[x_i]_{\{i|\langle d, \rho_i\rangle = 0\}}.$$  The last equality follows from the fact that the $d(\CC^*)$-invariant polynomials in $\CC[x_i]_{\{i|\langle d, \rho_i\rangle \geq 0\}}$ are exactly the polynomials which are homogeneous of degree zero with respect to the grading: $\deg_d(x_i) :=  \langle d, \rho_i\rangle$.  Thus $$\CC[x_i]_{\{i|\langle d, \rho_i\rangle \geq 0\}}^G = \CC[x_i]_{\{i|\langle d, \rho_i\rangle = 0\}}^G.$$ Combining this with the previous paragraph concludes the proof.
\end{proof}

Next we give conditions for when a derivative $\left(\prod_{i \in {\hat I}}z\partial_{\rho_i} \right) \II^Y$ is supported in cohomology of compact type.
\begin{proposition}\label{p:narrowI}
Let ${\hat I} \subset \{1, \ldots, r\}$ be such that $ \CC[x_i]_{\{i \notin {\hat I}\}}^G = \CC$.  Then 
$\left(\prod_{i \in {\hat I}}z\partial_{\rho_i} \right) \II^Y$ lies in \\$H^*_{CR, ct}(Y)[[q]][[t^i]][z, z^{-1}]$.
\end{proposition}

\begin{proof}  Fix $d \in   \hom_\ZZ \left( \widehat G, \QQ\right)$, $k \geq 0$, and $g \in G$.  Let $QG_{0, (g)|k} (Y, d)$
 denote the open and closed subset of 
$QG_{0, 1|k} (Y, d)$ which maps to the twisted sector $Y_g$ under $ev_1$.
Let
$F^{k,d}_{g,0}$ denote the intersection of
the fixed locus $F^{k,d}_{1,0}$  with  $QG_{0, (g)|k} (Y, d)$.
We will prove the statement for each the summand
$$\II^Y_{d,k,g} := \wt{ev_1}_*\op{Res}_{F^{k,d}_{g,0}}(\bt^k)$$
of $\II^Y(q, \bt, z)$.

Let $I(g) \subset \{1, \ldots, r\}$ denote the set of indices such that $g$ acts trivially on $\CC_{\rho_i}$, i.e. $\rho_i(g) = 1.$ The twisted sector $Y_g$ is equal to the vanishing locus $\{x_i = 0\}_{i \notin I(g)} \subset Y$.  The
 map $QG_{0, (g)|k} (Y, d) \to \spec \left( \CC[x_i]_{\{1\leq i\leq r\}}^G\right)$ factors as
 $$QG_{0, (g)|k} (Y, d) \to Y_g \to \spec \left( \CC[x_i]^G_{i \in I(g)}\right) \to \spec \left( \CC[x_i]_{\{1\leq i\leq r\}}^G\right).$$

Let ${\hat I}^0_d(g) \subset {\hat I}$ denote the set of indices $i \in {\hat I} \cap I(g)$ such that $\langle d, \rho_i\rangle = 0$.
and let ${\hat I}^1_d(g) \subset {\hat I}$ denote the set of indices $i \in {\hat I}\cap I(g)$ such that $\langle d, \rho_i\rangle \neq 0$.  Using the explicit formula for $\II^Y$ given in \cite{CCKbig}, we observe that for $i \in {\hat I}^0_d(g)$, 
$$\partial_{\rho_i} \II^Y_{d,k,g} = \rho_i \II^Y_{d,k,g}.$$ Let $Z_{{\hat I}^0_d(g)} := \{x_i = 0\}_{i \in {\hat I}^0_d(g)} \subset IY$ and let $s_0\colon Z_{{\hat I}^0_d(g)} \to IY$ denote the inclusion.  
Consider the commutative diagram
\[
\begin{tikzcd}[cramped,row sep=tiny, column sep=0 cm]
F_{{\hat I}^0_d(g)} \ar[dd, "c_{{\hat I}^0_d(g)}"] \ar[dr] \ar[rr, "(\wt{ev_1}  \circ \iota_d)'"] & & Z_{{\hat I}^0_d(g)} \ar[dd] \ar[dr, shorten >=1.85ex, "s_0"] & \\
& F^{k,d}_{g,0} \ar[dd] \ar[rr, crossing over, near start, "\wt{ev_1} \circ \iota_d"] & & Y_g \ar[dd, "c_g"] \\
\spec \left( \CC[x_i]^G_{i \in I(g) \setminus {\hat I}}\right) \ar[dr] \ar[rr]  & & \spec \left( \CC[x_i]^G_{i \in I(g) \setminus {\hat I}^0_d(g)}\right) \ar[dr] \\
& \spec \left( \CC[x_i]^G_{i \in I(g) \setminus {\hat I}^1_d(g)}\right) \ar[rr] \ar[uu, leftarrow, crossing over]& &\spec \left( \CC[x_i]^G_{i \in I(g)}\right),
\end{tikzcd}
\]
where $F_{{\hat I}^0_d(g)}$ is the fiber product of $\wt{ev_1} \circ \iota_d$ and $s_0$.  Then
\begin{align} \label{e:properstring}
\left(\prod_{i \in {\hat I}^0_d(g)}\partial_{\rho_i}\right) \II^Y_{d,k,g} &= \left(\prod_{i \in {\hat I}^0_d(g)} \rho_i \right)\II^Y_{d,k,g} \\ \nonumber
&= {s_0}_*{s_0}^* \II^Y_{d,k,g}  
\\ \nonumber
&= {s_0}_*{s_0}^* (\wt{ev_1} \circ {\iota_d})_*\op{Res}_{F^{k,d}_{g,0}}(\bt^k)  \\ \nonumber
&= {s_0}_* (\wt{ev_1} \circ {\iota_d})'_* {s_0}^!  \op{Res}_{F^{k,d}_{g,0}}(\bt^k) . 
\end{align}
The map $c_{{\hat I}^0_d(g)}$ is seen to be proper because $\wt{ev_1} \circ {\iota_d}$, $s_0$, and $c_g$ are \cite[\href{https://stacks.math.columbia.edu/tag/01W6}{Lemma 01W6}]{stacks-project}.  And since $\spec \left( \CC[x_i]^G_{i \notin {\hat I}}\right)$ is assumed to be a point, so is $\spec \left( \CC[x_i]^G_{i \in I(g) \setminus {\hat I}}\right)$.  We conclude that  $F_{{\hat I}^0_d(g)}$ is proper.  

Finally, \eqref{e:properstring} gives us
\begin{align*}
\left(\prod_{i \in {\hat I}}z\partial_{\rho_i} \right) \II^Y_{d,k,g} & = z^{|{\hat I}|} \prod_{i \in {\hat I} \setminus {\hat I}^0_d(g)}(\partial_{\rho_i}) \prod_{i \in {\hat I}^0_d(g)}(\partial_{\rho_i}) \II^Y_{d,k,g} \\
& = z^{|{\hat I}|} \prod_{i \in {\hat I} \setminus {\hat I}^0_d(g)}(\partial_{\rho_i})  {s_0}_* (\wt{ev_1} \circ {\iota_d})'_* {s_0}^!  \op{Res}_{F^{k,d}_{g,0}}(\bt^k)  \\
&= {s_0}_*(\wt{ev_1} \circ {\iota_d})'_* {s_0}^!  \left(z^{|{\hat I}|} \prod_{i \in {\hat I} \setminus {\hat I}^0_d(g)}(\partial_{\rho_i}) \op{Res}_{F^{k,d}_{g,0}}(\bt^k) \right). 
\end{align*}

This exhibits $\left(\prod_{i \in {\hat I}}z\partial_{\rho_i} \right) \II^Y_{d,k,g} $
as the pushforward of a class on  the proper stack $F_{{\hat I}^0_d(g)}$.  
We conclude that $\left(\prod_{i \in {\hat I}}z\partial_{\rho_i} \right) \II^Y_{d,k,g} $ lies in $ H^*_{ct}(Y_g)[[q]][[t^i]][z, z^{-1}] $.
%By the previous Proposition, $\left(\prod_{i \in {\hat I}}z\partial_{\rho_i} \right) \II^Y$ is represented by a cycle lying over $ \CC[x_i]_{\{i \notin {\hat I}\}}^G = \CC$.  Since $$IY \times_{\spec \left( \CC[x_1, \ldots, x_r]^G\right)} \spec \left( \CC[x_i]_{\{i \notin {\hat I}\}}^G\right)$$ is proper, we conclude that $\left(\prod_{i \in {\hat I}}z\partial_{\rho_i} \right) \II^Y$ is represented by a cycle with proper support.  The conclusion follows.
\end{proof}
%
%
%%idea: always differentiate in the $id_{LG, j}$ direction (i.e. the class of the $\CC^*_R$ invariant part of the identity sector in GW).

\begin{theorem}[$I$-functions for toric GLSMs]\label{t:GLSMI}
Let ${\hat I} \subset \{1, \ldots, r\}$ be such that $ \CC[x_i]_{\{i \notin {\hat I}\}}^G = \CC$.  Let $\rho_{\hat I} := \prod_{i \in {\hat I}} \rho_i$.
Then 
$\sigma^w (z\partial_{\rho_{\hat I}} \II^Y)$ is a $(0^+, 0^+)$ GLSM $I$-function.
\end{theorem}

\begin{proof}
From the explicit description of the $\II^Y$, we observe that $z\partial_{\rho_{\hat I}} \II^Y = \left(\prod_{i \in {\hat I}}z\partial_{\rho_i} \right) \II^Y$. 
 The theorem then follows immediately from Theorem~\ref{t:dI} and Proposition~\ref{p:narrowI}.
\end{proof}

Let $c_i$ denote the degree of $x_i$ with respect to the $\CC^*_R$-action.
A natural choice of ${\hat I}$ is given by ${\hat I} = \{i | c_i \neq 0\}$.  
In this case, $\CC[x_i]_{i \notin {\hat I}} = \left(\CC[x_i]_{1\leq i\leq r}\right)^{\CC^*_R}$, and $\CC[x_i]_{i \notin {\hat I}}^G = \left(\CC[x_i]^{\CC^*_R}_{1\leq i\leq r}\right)^G = \left(\CC[x_i]_{1\leq i\leq r}\right)^\Gamma$.
\begin{corollary}\label{c:Ifctn}
If $\left(\CC[x_i]_{1\leq i\leq r}\right)^\Gamma = \CC$, then 
\begin{align} \nonumber
\II^{(Y, w)}(\bt, q, z) :=& \sum_{d \in\eff(V, G, \theta)}q^d \exp \left(  \frac{1}{z} \sum_{j=1}^l t^j p_j\left(\eta_s + z\langle d, \eta_s\rangle\right) \right) \\
&\sigma^w\left( %\prod_{i| c_i \neq 0} 
\frac{\prod_{i| c_i \neq 0,  \langle d, \rho_i\rangle\leq 0} \prod_{ \langle d, \rho_i\rangle \leq \nu\leq 0} (\rho_i + ( \langle d, \rho_i\rangle - \nu)z)}
{\prod_{i| c_i \neq 0, \langle d, \rho_i\rangle> 0} \prod_{ 0 < \nu < \langle d, \rho_i\rangle} (\rho_i + ( \langle d, \rho_i\rangle - \nu)z)}  \nonumber \right. \\
&\left. %\prod_{i| c_i = 0}
 \frac{\prod_{ i| c_i = 0, \langle d, \rho_i\rangle< 0} \prod_{ \langle d, \rho_i\rangle \leq \nu< 0} (\rho_i + ( \langle d, \rho_i\rangle - \nu)z)}
{\prod_{ i| c_i = 0, \langle d, \rho_i\rangle> 0} \prod_{ 0 \leq \nu < \langle d, \rho_i\rangle} (\rho_i + ( \langle d, \rho_i\rangle - \nu)z)} \ii_{g_d^{-1}} \right) \label{e:IYW}
\end{align}
is a GLSM $I$ function for $(Y, w)$.
\end{corollary}

\begin{proof}
The function $\II^{(Y, w)}(\bt, q, z)$ of \eqref{e:IYW} is $\sigma^w (z\partial_{\rho_{\hat I}} \II^Y(\bt, q, z))$ for ${\hat I} = \{i | c_i \neq 0\}$.  The result is an application of Theorem~\ref{t:GLSMI}.
\end{proof}

\section{Examples and comparisons}\label{s:EXS}

In this section we give explicit formulas for $(0^+, 0^+)$ $I$-functions for particular GLSMs.  We show that for many GLSMs where mirror theorems have already been proven, the $(0^+, 0^+)$ $I$-functions which arise from Theorem~\ref{t:GLSMI} agree with previously computed $I$-functions, including in many instances where concavity/convexity fail.

\subsection{FJRW theory}
The first example of a GLSM is the case that $G$ is finite and $w\colon [\CC^n/ G] \to \CC$ is a quasihomogeneous polynomial with an isolated singularity at the origin.  These are sometimes called affine phase GLSMs.  The corresponding invariants are known as FJRW invariants and were defined by Fan--Jarvis--Ruan \cite{FJR1} prior to defining enumerative invariants for more general GLSMs.  FJRW mirror theorems have been proven in the case when $w$ is a Fermat polynomial \cite{ChR, ChIR, LPS} and in the case that $w$ is a chain polynomial and $ G$ is maximal \cite{Gu}.

Here we compute a $(0^+, 0^+)$ GLSM $I$-function for a GLSM representing the LG model $(Y, w)$ in the case when $Y = [\CC^n /  G]$ and $ G$ is a finite subgroup of $(\CC^*)^n$. Assume $w = w(x_1, \ldots, x_n)$ is homogeneous of degree $d_w$ with respect to a $\CC^*_R$-action, where $\CC^*_R$  acts with weight $c_i$ on $x_i$.  Under the assumption that $Z(dw)$ is proper over $\spec \CC$, $w$ has an isolated singularity at the origin of $\CC^n$.

 The first step is to present $Y$ as a GIT quotient by a torus:  $$Y = [\CC^{n+s} \sslash_\theta (\CC^*)^{s}].$$  The simplest case is when $G = \langle J \rangle$.  Let $\wt G = \CC^*$ act on $\wt V = \spec \CC[x_1, \ldots, x_n, p]$ with weights $(c_1, \ldots, c_n, -d_w)$.  Then $\wt w = p \cdot w$ is $\wt G$-invariant.  We extend the $\CC^*_R$-action from $\CC^n$ to $\wt V$ by letting $\CC^*_R$ act trivially on the $p$ coordinate.  Choose $\theta = (-1)$, then the GLSM $(\wt V, \wt G, \theta, \wt w)$ represents the LG model $(Y, w)$.
 
Let $\eta_{p}\colon \wt G \to \CC^*$ denote the homomorphism of weight $-d_w$, viewed as an element of $H^*([\wt V/\wt G])$.  Define $\bt = t \cdot \eta_{p}$ in the notation of \eqref{e:IY}.  Then the corresponding big $I$-function of \eqref{e:IY} is given by:
\begin{align*} 
\II^Y(\bt, q, z) = \sum_{k \geq 0} \frac{q^{k/d_w} e^{tk} \prod_{i=1}^n  \prod_{0<  \nu \leq c_i k/d_w} z(-c_i k/d_w + \nu) }{z^k k!} \ii_{e^{2 \pi i  k/d_w}}
\end{align*}
When $Y$ is a finite quotient of affine space, the map $\phi^w\colon \cc H_{ct}(Y, w) \to H^*_{CR,ct}(Y)$ is an isomorphism, so $\sigma^w  = (\phi^w)^{-1}$.  The class $\ii_{e^{2 \pi i  k/d_w}}$ is of compact type if and only if  $$\{i |  c_i k/d_w  \in \ZZ\} = \emptyset.$$ For such $k$, let  $\varphi_{k-1} = \sigma^w(\ii_{e^{2 \pi i  k/d_w}})$ (note the shift in index).

We observe that $\CC[x_1, \ldots, x_n]^G = \CC$.  Therefore, by Theorem \ref{t:GLSMI}, a $(0^+, 0^+)$ GLSM $I$-function is obtained as $$\II^{(Y, w)}(\bt, q, z) := \sigma^w (z\partial_{\eta_p} \II^Y)= \sigma^w(z \frac{\partial}{\partial t}\II^Y).$$  Here one observes that the coefficient of $q^{k/d_w}$ in $\II^Y(\bt, q, z)$ is zero if $k>0$ and $\{i |  c_i k/d_w  \in \ZZ\} \neq \emptyset$.  It follows that $z \frac{\partial}{\partial t}\II^Y$ is supported in $ H^*_{CR,ct}(Y)$, thus verifying Proposition~\ref{p:narrowI} directly in this case.  Explicitly, we have 
\begin{equation}\label{e:IFJRW} \II^{(Y, w)}(\bt, q, z) = \sum_{\substack{ k \geq 1 \\ \{i |  c_i k/d_w  \in \ZZ\} = \emptyset}} \frac{q^{k/d_w} e^{tk} \prod_{i=1}^n  \prod_{0<  \nu \leq c_i k/d_w} z(-c_i k/d_w + \nu) }{z^{k-1} (k-1)!}  \varphi_{k-1}.
\end{equation}
This function agrees with previously computed FJRW $I$-functions, for instance with the non-equivariant limit of \cite[equation (40)]{ChIR} up to the substitution $u = -q^{1/d_w}e^t$ and a factor of $-z$.  In the case that $w$ is a \emph{Fermat} polynomial, the mirror theorem of Chiodo--Iritani--Ruan \cite[Theorem~3.10]{ChIR} implies that \eqref{e:IFJRW} is equal to the FJRW $J$-function after a change of variables and a scaling.  From our perspective, this may be interpreted as an $\epsilon$-wall crossing result from $(0^+, 0^+)$-stability to $(\infty, \infty)$-stability for the case that $w$ is Fermat.  We emphasize, however, that \eqref{e:IFJRW} is a generating function of $(0^+, 0^+)$-stable GLSM invariants for \emph{any} choice of quasihomogeneous $w$ with isolated singularity at the origin.

For a more general finite group $G$, write $G$ as a product of cyclic groups, $G = \prod_{j=1}^s \mu_{r_j}$.  Choose generators $\omega_j \in \mu_{r_j}$ and suppose 
$$\omega_j \cdot (x_1, \ldots, x_n) = (e^{2 \pi i c_{1j}/r_j} x_1, \ldots , e^{2 \pi i c_{nj}/r_j} x_n).$$ For notational simplicity let us assume that $\omega_1 = J$, so $r_1 = d_w$.  \footnote{This assumption sacrifices some generality but simplifies the formulas and is sufficient for the purpose of comparison with previous works.  We leave the details of the most general case to the interested reader.}

Define $\wt V = \spec \CC[x_1, \ldots, x_n, p_1, \ldots, p_s]$ by $\wt G = (\CC^*)^s$, where the $j$th factor of $\CC^*$ acts on $\wt V$ with weights $$(c_{1j}, \ldots, c_{nj}, 0 \ldots,  -r_j,  \ldots, 0),$$ where the term $-r_j$ occurs in the $n+j$th coordinate.
Then $[\CC^n /G]$ may be expressed as a GIT stack quotient $[\wt V \sslash_\theta  \wt G]$, where 
$$\theta (\lambda_1, \ldots, \lambda_s) = \left( \prod_{j=1}^s \lambda_j \right)^{-1}.$$  If we define $\wt w$ to be $p_1\cdots p_s w(x_1, \ldots, x_n)$, then $(\wt V, \wt G, \theta, \wt w)$ is a GLSM representing the LG model $(Y, w)$.

Define $\eta_{p_1}\colon \wt G \to \CC^*$ by $\eta_{p_1} (\lambda_1, \ldots, \lambda_s) = \lambda_1^{-r_1}$ and let 
$\bt = t \cdot \eta_{p_1} \in H^*([\wt V/\wt G])$.  In this case the big $I$-function of \cite{CKbig, CCKbig} (as well as the ``S-extended stacky $I$-function'' of \cite{CCIT2}) is:
\begin{align}
\nonumber &\II^Y(\bt, q, z) \\ =& \nonumber
 \sum_{d_1, \ldots, d_s \geq 0} \frac{ e^{td_1} \prod_{j=1}^s q_j^{d_j/r_j}}{ \prod_{j=1}^sd_j!  z^{d_j}} \left( \prod_{i=1}^n \prod_{0< \nu \leq \sum_{j=1}^s c_{ij} d_j/r_j } z( - \sum_{j=1}^s c_{ij} d_j/r_j + \nu)\right) \ii_{J^{d_1} \prod_{j=2}^s \omega_j^{d_j}}.
\end{align}
As before we note that for $(d_1, \ldots, d_s) \neq (0, \ldots, 0)$, the coefficient of $\prod_{j=1}^s q_j^{d_j}$ is zero unless 
$\{i | \sum_{j=1}^s c_{ij} d_j/r_j \in \ZZ \} = \emptyset$, i.e., unless $ \ii_{J^{d_1} \prod_{j=2}^s \omega_j^{d_j}} \in H^*_{CR, ct}(Y)$.  Therefore $z \frac{\partial}{\partial t} \II^Y(q, \bt, z) $ lies in compact type cohomology.

Then by Theorem~\ref{t:dI}, 
a $(0^+, 0^+)$ GLSM $I$-function  is given by $\sigma^w( z \frac{\partial}{\partial t} \II^Y(q, \bt, z) )$:
\begin{align} \nonumber
\II^{(Y, w)}(\bt, q, z) & = 
 \sum_{\substack{ d_1\geq 1 \\ d_2, \ldots, d_s \geq 0 \\ \{i | \sum_{j=1}^s c_{ij} d_j/r_j \in \ZZ \} = \emptyset}} \left( \frac{ 
  e^{td_1} \prod_{j=1}^s q_j^{d_j/r_j}}{ (d_1-1)!z^{d_1 - 1} \prod_{j=2}^s d_j! z^{d_j}} \right.\\
 & \left. \prod_{i=1}^n \prod_{0< \nu \leq \sum_{j=1}^s c_{ij} d_j/r_j } z( - \sum_{j=1}^s c_{ij} d_j/r_j + \nu)\right) \sigma^w\left(  \ii_{J^{d_1} \prod_{j=2}^s \omega_j^{d_j}}\right). \label{e:IGu}
\end{align}
Again this formula agrees with FJRW $I$-functions previously computed in special cases (again up to a simple change of variables and scaling).  For $w$ a Fermat polynomial but $G$ no longer necessarily cyclic, \eqref{e:IGu} appears in \cite{AS1} (based on the work in \cite{LPS}).  
The restriction of \eqref{e:IGu} to $q_2 = \cdots = q_s = 0$ recovers the FJRW $I$-function for \emph{chain} polynomials with $G = G_{max}$ \emph{maximal} as computed by Gu\'er\'e in \cite[Equation~(98)]{Gu}. \footnote{There is a very minor typo in \cite[Equation~(98)]{Gu}, which we have confirmed with the author.  There, the symbol $\delta_j$ should be defined as 
$\delta_j := - \delta_{\left\{ \substack{ N-j \text{ is odd,} \\ \bq_j k \in \ZZ }\right\}}.$  It is with this modification that \eqref{e:IGu} agrees with \cite[Equation~(98)]{Gu}.}

\subsection{Complete intersections}
Let $X$ be a smooth orbifold with projective coarse moduli space arising as a toric GIT stack quotient $X = [V_1 \sslash_\theta G]$.  An $I$-function for $X$, $\II^X(\bt, q, z)$, was given in \eqref{e:IY}.  Choose characters $\tau_j \in \widehat G$ for $1 \leq i \leq n$, and $G$-equivariant functions $f_j(\underline x)\colon V_1 \to \CC_{\tau_j}$.  Assume that the intersection of  $Z(f_1, \ldots, f_n)  \subset V$ with the $\theta$-semistable locus $V_1^{ss}$ is smooth of codimension $n$.  Let $Z$ denote the corresponding complete intersection $[Z (f_1, \ldots, f_n) \cap V_1^{ss} / G]$ in $X$.  Let $L_{\tau_j} \to X$ denote the line bundle associated to the character $\tau_j$.  The quantum Lefschetz hyperplane theorem \cite{KKP, CKM} allows one to obtain an $I$-function for $Z$ from an $I$-function for the ambient space $X$ whenever the line bundles $L_{\tau_j}$ are each \emph{convex}.  If $X$ is a smooth variety, $L_{\tau_j}$ is convex whenever it is \emph{semi-positive}, however this is no longer the case when $X$ is an orbifold.  This failure of convexity has presented a significant obstacle in proving a  quantum Lefschetz statement for orbifold complete intersections \cite{CGIJJM}, however such a statement has recently been proven by Wang \cite{Wang} when $X$ is toric.

As described by Witten \cite{Wi2}, one can associate to the above data a GLSM, $(V, G, \theta, w)$.  Let $V_2 = \spec \CC[p_1, \ldots, p_n]$ be the $n$-dimensional representation of $G$ such that the action of $G$  on the $j$th factor is given by the dual of $\tau_j$.  Let $V = V_1 \times V_2$ and define a potential $w\colon V \to \CC$ by
$$w = \sum_{j=1}^n p_j \cdot f_j(\underline x).$$
By construction $w$ is $G$-invariant.  Let $\CC^*_R$-act on $V$ with weight $0$ on $V_1$ and weight $1$ on $V_2$.  Then $w$ is homogeneous of degree $1$.  Assume that each $L_{\tau_j}$ is semipositive in the sense of \cite[Definition~2.6]{Wang}, and that $V^{ss} = V_1^{ss} \times V_2$, so that $[V \sslash_\theta G] \to [V_1 \sslash_\theta G]$ is a vector bundle.  Then $(V, G, \theta, w)$ is a GLSM which is expected to correspond in some sense to the complete intersection $Z$, in particular the GLSM invariants of $(V, G, \theta, w)$ should agree with the Gromov--Witten/quasimap invariants of $Z$ up to a sign. 

Let $Y = [V \sslash_\theta G]$.  
For $g \in G$ and $\tau_j$ a character, define $\iota_g(\tau_j)$ to be the rational number such that for $l \in \CC_{\tau_j}$, $$g \cdot l = e^{2 \pi i \iota_g(\tau_j)}l$$ and $0 \leq \iota_g(\tau_j) <1$.
By \cite[Lemma~2.14]{HS}, the compact type Gromov--Witten state space $H^*_{CR,ct}(Y)$ is equal to $\op{im}(i_*)$ where $i\colon IX \to IY$ is the inclusion of the zero section, and 
is generated as an $H^*(Y)$-module by classes of the form $e(E^\vee_g) \ii_g$, where $E_g$ denotes the (pullback to $IY$ of the) direct sum 
$$ \bigoplus_{\{j| \iota_g(\tau_j) = 0\}} L_{\tau_j}.$$

Let $j\colon IZ \to IX$ denote the closed immersion of inertia stacks.  Assume that the Chen--Ruan Poincar\'e pairing on the \emph{ambient cohomology}
$$H^*_{CR, amb}(Z) := \op{im}(j^*)$$
is nondegenerate.  This is equivalent to the condition that $H^{\op{even}}(IZ) = \op{ker}(j_*) \oplus \op{im}(j^*)$ \cite{IMM}, and holds, for instance, if each line bundle $L_{\tau_j}$ is ample.  

Under the above assumption, by \cite[Lemmas~6.5, 6.7]{HS} there exists an isomorphism 
$$\wt \Delta\colon H^*_{CR,ct}(Y) \to H^*_{CR, amb}(Z)$$
characterized by the fact that for $\alpha \in H^*(X_g) \subset H^*(IX)$, 
$$\wt \Delta (i_*(\alpha)) = e^{\pi i \sum_{j=1}^n \iota_g(\tau_j)} j^*(\alpha).$$
In this case, the composition $$\wt \Delta \circ \phi^w\colon \cc H_{ct}(Y, w) \to H^*_{CR, amb}(Z)$$ is an isomorphism when restricted to 
$ \cc R_c(Y, w)$.  We will use this state space identification to compare the GLSM $I$-functions obtained by Corollary~\ref{c:Ifctn} with the $I$-functions computed by Wang in \cite[Section 3.1]{Wang}.  

Begin with the $I$-function for $Y$ from \eqref{e:IY}.  One easily checks that 
\begin{align}\label{IYX}
\II^Y(\bt, q, z) =& \sum_{d \in\eff(V, G, \theta)}q^d \II^X_d(\bt, z)
\prod_{j=1}^n  \prod_{ 0< \nu \leq \langle d, \tau_j \rangle } (-\tau_j + ( - \langle d, \tau_j \rangle + \nu)z) \ii_{g_d^{-1}}
\end{align}
where $\II^X_d(\bt, z)$ is the factor in front of $ \ii_{g_d^{-1}}$ in the $q^d$ coefficient of $\II^X(\bt, q, z)$.
By Corollary~\ref{c:Ifctn}, the expression $\II^{(Y, w)}(\bt, q, z) :=
\sigma^w\left(\prod_{j=1}^n (z\partial_{\tau_j})\II^Y(\bt, q, z)\right)$ is an $I$-function for the GLSM $(Y, w)$.

After applying $\prod_{j=1}^n (z\partial_{\tau_j})$ to \eqref{IYX} we obtain a more explicit expression
\begin{align*}\II^{(Y, w)}(\bt, q, z) =&
\sigma^w\left(\sum_{d \in\eff(V, G, \theta)}q^d \II^X_d(\bt, z)
\prod_{j=1}^n  \prod_{ 0\leq \nu \leq \langle d, \tau_j \rangle } (-\tau_j + ( - \langle d, \tau_j \rangle + \nu)z) \ii_{g_d^{-1}}\right) \\
=& \sigma^w\left(\sum_{d \in\eff(V, G, \theta)}q^d \II^X_d(\bt, z)
\prod_{j=1}^n  \prod_{ 0\leq \nu < \langle d, \tau_j \rangle } (-\tau_j + ( - \langle d, \tau_j \rangle + \nu)z) e(E_{g_d^{-1}}^\vee) \ii_{g_d^{-1}}\right) 
\end{align*}
Pulling out the negative signs from each factor, the product $\prod_{ 0\leq \nu < \langle d, \tau_j \rangle } (-\tau_j + ( - \langle d, \tau_j \rangle + \nu)z)$ may be written as
\begin{align*}
 &
\prod_{j=1}^n  \prod_{ 0\leq \nu < \langle d, \tau_j \rangle } (-1)^{\lceil  \langle d, \tau_j \rangle  \rceil } (\tau_j + ( \langle d, \tau_j \rangle - \nu)z)  \\
=& 
\prod_{j=1}^n  \prod_{ 0\leq \nu < \langle d, \tau_j \rangle } e^{-\pi i (  \langle d, \tau_j \rangle   + \iota_{g_d^{-1}}( \tau_j)) } (\tau_j + ( \langle d, \tau_j \rangle - \nu)z) 
\end{align*}
After applying the change of variables $q^d \mapsto q^d e^{\pi i \sum_j\langle d, \tau_j \rangle}$ and the linear map $\wt \Delta \circ \phi^w$,
 we recover
  \begin{align*}&\wt \Delta \circ \phi^w \left(\II^{(Y, w)}(\bt, q, z)\right)|_{q^d \mapsto q^d e^{\pi i \sum_j \langle d, \tau_j \rangle}} \\
  =& 
j^* \circ i^* \left(\sum_{d \in\eff(V, G, \theta)}q^d \II^X_d(\bt, z)
\prod_{j=1}^n  \prod_{ 0\leq \nu < \langle d, \tau_j \rangle }  (\tau_j + ( \langle d, \tau_j \rangle - \nu)z) 
 \ii_{g_d^{-1}}\right).
 \end{align*}
This coincides exactly with the $I$-function for $Z$ obtained by Wang in \cite[Section 3.1]{Wang}.  Thus $\wt \Delta \circ \phi^w|_{\cc R_c(Y, w)}$ identifies $\II^{(Y, w)}(\bt, q, z)$ with the $I$-function for $Z$ after the change of variables $q^d \mapsto q^d e^{\pi i \sum_j\langle d, \tau_j \rangle}$.

\subsection{Hybrid models}
Going beyond Gromov--Witten and FJRW theory, one of the few examples of a class of GLSM where mirror theorems have been proven are so-called \emph{hybrid models}, which simultaneously generalize both the Gromov--Witten theory of hypersurfaces and FJRW theory.  These GLSMs arise when one starts with a GLSM representing a complete intersection as in the previous section, but chooses a new stability condition $\theta_-$.  These were computed by Clader--Ross in \cite{ClRo, CRwall}.

More precisely, let $G = \CC^*$ act on $V = \spec \CC[x_1, \ldots, x_m, p_1, \ldots, p_n] \cong \CC^{m + n}$ with weights $(w_1, \ldots, w_m, -d_1, \ldots, -d_n)$.  Choose polynomials $f_1(\underline x), \ldots, f_n(\underline x)$ such that 
$$Z(f_1, \ldots, f_n) \subset \PP(w_1, \ldots, w_m)$$ 
is a smooth complete intersection.  Define $\theta_{+/-} = (+/- 1) \in \widehat G \cong \ZZ$ and let $w = \sum_{j=1}^n p_j f_j(\underline x)$.  If we let $\CC^*_R$ act with weight 0 on $x_i$ and 1 on $p_j$ for each $1 \leq i \leq m$, $1 \leq j \leq n$, then $w$ is homogeneous of degree 1.
The GLSM $(V, G, \theta_+, w)$ represents the complete intersection $Z(f_1, \ldots, f_n)$ as evidenced above.  The GLSM $(V, G, \theta_-, w)$ is called a hybrid model.  
In this case,
$$Y := [V \sslash_{\theta_-} G] = \bigoplus_{i=1}^m \cc O_{\PP(d_1, \ldots, d_n)}(-w_i).$$

Let $H$ denote the class $c_1(\cc O_{\PP(d_1, \ldots, d_n)}(1))$ supported on the untwisted sector $H^*(Y) \subset H^*_{CR}(Y)$ (here we use  $\cc O_{\PP(d_1, \ldots, d_n)}(1)$ to denote the line bundle $[(V \times \CC_{\theta_-}) \sslash_{\theta_-} G]$ corresponding to character of $\widehat G$ with weight $-1$).  
Let $\eta_{p_j}(\lambda) = \lambda^{-d_j}$ and let $\bt = \sum_{j=1}^n t_j \eta_{p_j} \in H^*([V/G])$.  Let $d = \op{lcm}(d_1, \ldots, d_n)$.
Then the big $I$-function of $Y$ is given by
\begin{align}
\II^Y(\bt, q, z) &=
 \sum_{k \geq 0} q^{k/d}\left( \prod_{j=1}^n e^{t_j(d_j  H/z + d_jk/d)}  \right. \nonumber \\
 &\left. \frac{\prod_{i=1}^m \prod_{0< \nu \leq w_i k/d} \left( -w_i H + ( - w_i k/d + \nu)z\right)} {\prod_{j=1}^n \prod_{0\leq \nu < d_jk/d}. \left( d_j H + (d_jk/d - \nu) z\right)}  \ii_{e^{2 \pi i k/d}}\right).
\end{align}
Because $\CC[x_1, \ldots, x_m]^G = \CC$, the function $$\left(\prod_{j=1}^n z \partial_{\eta_j}\right) \II^Y(\bt, q, z) = \left(\prod_{j=1}^n z \frac{\partial}{\partial t_j} \right) \II^Y(\bt, q, z)$$ lies in $H^*_{CR, ct}(Y)$.  Therefore an $I$-function for $(V, G, \theta_-, w)$ is obtained as  
\begin{align}
\II^{(Y, w)}(\bt, q, z) &:= \sigma^w\left( \left(\prod_{j=1}^n z \frac{\partial}{\partial t_j} \right) \II^Y(\bt, q, z)\right) \nonumber \\
&= \sum_{k \geq 0} q^{k/d}\sigma^w \left( \prod_{j=1}^n e^{t_j(d_j  H/z + d_jk/d)}  \right. \nonumber \\
 &\left. \frac{\prod_{i=1}^m \prod_{0< \nu \leq w_i k/d} \left( -w_i H + ( - w_i k/d + \nu)z\right)} {\prod_{j=1}^n \prod_{0< \nu < d_jk/d}. \left( d_j H + (d_jk/d - \nu) z\right)}  \ii_{e^{2 \pi i k/d}}\right). \label{e:ICR}
\end{align}
This recovers the hybrid-model $I$-function computed by Clader--Ross in \cite{ClRo, CRwall}.  Their $I$-function appears, for instance, as $I^{X_-, W}$ in \cite[Section~7.4]{ClRo}, in the proof of Lemma~7.6.  Specializing $t_1 = \cdots = t_n = 0$ in \eqref{e:ICR} recovers this function.\footnote{There is a very minor typo in the formula for $I^{X_-, W}$ in \cite{ClRo}, the factors $(-bz + w_i H)$ should in fact be $(-bz - w_i H)$.}

There are two assumptions, denoted (A1) and (A2) required for the mirror theorem in \cite{ClRo}.  The first implies a version of concavity for the hybrid model invariants, the second guarantees that the $I$-function is supported in the \emph{narrow} cohomology (a stronger condition than being compact type).  However equation~\eqref{e:ICR} gives a $(0^+, 0^+)$ $I$-function regardless of these assumptions. 

\begin{acknowledgements}
I would like to thank H. Fan, T. Jarvis, and Y. Ruan for explaining the gauged linear sigma model and patiently answering many questions.  Several ideas in this paper are inspired by collaboration with I. Ciocan--Fontanine, J. Gu\'er\'e, D. Favero, and B. Kim. I am very grateful to them for all that I learned during our time working together.  I  also thank K. Aleshkin, E. Clader, J. Knapp, Y.-P. Lee, M. Liu, N. Priddis, M. Romo, D. Ross,  Y. Shen, and G. Xu for helpful conversations on GLSMs and FJRW theory.  Finally, I thank the anonymous referees for their careful reading of the manuscript and many valuable comments.
%This work was partially supported by NSF grant DMS-1708104 and Simons Foundation Travel Grant 958189.
\end{acknowledgements}

\bibliographystyle{amsalpha}
\bibliography{references2}

\end{document}